\documentclass[11pt]{article}

\title{Asymptotic profiles of solutions\\ for structural damped wave equations}
\author{
Ryo Ikehata\thanks{ikehatar@hiroshima-u.ac.jp},\\
Department of Mathematics, \\
Graduate School of Education, Hiroshima University\\
Higashi-Hiroshima 739-8524, Japan \\
\ \\
Hiroshi Takeda\thanks{Corresponding author: h-takeda@fit.ac.jp},\\
Department of Intelligent Mechanical Engineering, \\
Faculty of Engineering, Fukuoka Institute of Technology, \\
3-30-1 Wajiro-higashi, Higashi-ku, Fukuoka 811-0295, Japan 
}
\date{}

\usepackage{amssymb,amsmath,amsthm}
\usepackage[dvips]{graphicx}
\usepackage{bm} 
\newcommand{\R}{\mathbb R}

\newcommand{\G}{\mathcal{G}}
\newcommand{\K}{\mathcal{K}}

\newcommand{\E}{\mathcal{E}}

\newcommand{\J}{\mathcal{J}}

\newcommand{\supp}{\mathop{\mathrm{supp}}\nolimits}
\renewcommand{\cosh}{\mathop{\mathrm{cosh}}\nolimits}
\renewcommand{\sinh}{\mathop{\mathrm{sinh}}\nolimits}

\newtheorem{thm}{Theorem}[section]
\newtheorem{cor}[thm]{Corollary}
\newtheorem{prop}[thm]{Proposition}
\newtheorem{lem}[thm]{Lemma}
\theoremstyle{remark}
\newtheorem{rem}[thm]{Remark}

\theoremstyle{definition}

\begin{document}
\maketitle

\numberwithin{equation}{section}
%%%%%%%%%%%%
\begin{abstract}
In this paper, we obtain several asymptotic profiles of solutions to 
the Cauchy problem for structurally damped wave equations 
$\partial_{t}^{2} u - \Delta u + \nu (-\Delta)^{\sigma} \partial_{t} u=0$, 
where $\nu >0$ and $0< \sigma \le1$.
Our result is the approximation formula of the solution by a constant multiple of a special function 
as $t \to \infty$,
which states that the asymptotic profiles of the solutions are classified into $5$ patterns 
depending on the values $\nu$ and $\sigma$.
\end{abstract}

\newpage
%%%%%%%%%%%%%%%%%%%%%%%%%%%%%%%%%%%%%%%%%%%%%%%%%%%%%%%%%%%%%%%%%%%%%%%%%%%%%%%%%%%%%%%%%%%%%%%%%%%%%%%%%%%%%%%%%%%%%%%%%%%%%%%%%%%%%%%%%%%%%%%	
\section{Introduction}
In this paper,
we consider the initial value problem for the following equations
\begin{equation} \label{eq:1.1}
\left\{
\begin{split}
& \partial_{t}^{2} u -\Delta u + \nu( -\Delta)^{\sigma} \partial_{t} u =0, \quad t>0, \quad x \in \R^{n}, \\
& u(0,x)=u_{0}(x), \quad \partial_{t} u(0,x)=u_{1}(x) , \quad x \in \R^{n}, 
\end{split}
\right.
\end{equation}
where $\sigma\in (0,1]$,
$\nu > 0$ is a constant,
$u_{0}(x)$ and $u_{1}(x)$ are
given initial data.\\

To begin with, let us introduce several related works to our problem \eqref{eq:1.1}. In the case when $\sigma = 1$ (i.e., strong damping case) one should make mention to some pioneering decay estimates of solutions due to Ponce \cite{P} and Shibata \cite{S}, in which Ponce \cite{P} dealt with rather special initial data such as $u_{1}(x) = \partial_{x}v(x)$ to avoid some singularity, and Shibata \cite{S} has established $L^{p}$-$L^{q}$ decay estimates of solutions. Karch \cite{K} studied an asymptotic self-similar profile of the solution as $t \to +\infty$ in the case when $\sigma \in [0,1/2)$, and Ikehata \cite{I-2} has derived total energy decay estimates of solutions to problem \eqref{eq:1.1} with $\sigma =1$ considered in the exterior of a bounded obstacle. While, Lu-Reissig \cite{LR} studied the parabolic effect in high order (total) energy estimates to problem \eqref{eq:1.1} with damping $\nu( -\Delta)^{\sigma} \partial_{t} u$ replaced by $b(t)( -\Delta)^{\sigma} \partial_{t} u$, and it seems that recent active researches concerning structural damped waves have their origin in \cite{LR}, however, in \cite{LR} they did not investigate any asymptotic profiles of solutions. Recently, Ikehata-Todorova-Yordanov \cite{ITY} have discovered its profile of solutions in asymptotic sense as $t \to +\infty$, and it should be mentioned that their result has been established as an abstract theory including \eqref{eq:1.1} with $\sigma = 1$, so that it includes quite wide applications. After \cite{ITY}, Ikehata \cite{I} re-studied the problem \eqref{eq:1.1} with $\sigma = 1$ to observe optimal decay estimates of solutions in terms of $L^{2}$-norm. The result of \cite{I} has its motivation in \cite{ITY}, and especially in Ikehata-Natsume \cite{IN}, in there they studied more precise decay estimates of the total energy and $L^{2}$-norm of solutions to the present problem \eqref{eq:1.1} by employing the energy method in the Fourier space developed by Umeda-Kawashima-Shizuta \cite{UKS}. Although the result of \cite{IN} has a gap near $\sigma = 0$, soon after \cite{IN} the gap has been completely embedded in Chara\~o-da Luz-Ikehata \cite{CLI} by developing a powerful tool to get energy decay estimates.  

While, quite recently, in a series of papers due to D'Abbicco \cite{D}, D'Abbicco-Ebert \cite{DE1, DE2, DE3}, D'Abbicco-Reissig \cite{DR} and Narazaki-Reissig \cite{NR} they have studied several decay estimates and asymptotic profiles of solutions to problem \eqref{eq:1.1} in terms of the $L^{p}$-norms ($1 \leq p \leq \infty$), but their main concern seems to be a little restrictive to the case for $0 \leq \sigma \leq 1/2$, i.e., a effective damping case of the problem \eqref{eq:1.1} is mainly studied, and so a non-effective damping aspect for the region $1/2 < \sigma \leq 1$ to problem \eqref{eq:1.1} seems to be less investigated at present. 

Our main purpose is to classify all asymptotic profiles of solutions to problem \eqref{eq:1.1} in terms of the constant $\nu$ and $\sigma$. Especially, our results below essentially seem new in the noneffective damping case for $\sigma \in (1/2, 1)$ as compared with a previous result due to D'Abbicco-Reissig \cite[Theorem 8]{DR}. In fact, our results below state about the asymptotic profile of the solution to problem \eqref{eq:1.1} in terms of the higher order derivatives, and as a result optimal decay order of the solution can be derived from the viewpoint of the higher order derivatives in $L^{2}$-sense.\\

To state our results, we introduce some notation, which will be used in this paper.

\begin{align*} 
%\label{eq:1.2}
\gamma_{\sigma, k}:=
\begin{cases}
& \frac{n}{4(1-\sigma)} -\frac{\sigma}{1-\sigma} 
+\frac{k}{2(1- \sigma)} \ \text{for}\ 0 \le \sigma < \frac{1}{2}, \\ 
& \frac{n}{2} +1-k \ \text{for}\ \sigma = \frac{1}{2}, \\
& \frac{n}{4\sigma} -\frac{1}{2\sigma} 
+\frac{k}{2\sigma}  \ \text{for}\ \frac{1}{2} < \sigma \le 1, 
\end{cases}
\end{align*}
\begin{align*} 
%\label{eq:1.3}
\tilde{\gamma}_{\sigma, k}
:=
\begin{cases}
& \frac{n}{4(1-\sigma)} 
-\frac{k}{2(1- \sigma)}  \ \text{for}\ 0 \le \sigma \le \frac{1}{2}, \\
& \frac{n}{2} -k \ \text{for}\ \sigma = \frac{1}{2}, \\
& \frac{n}{4\sigma} -\frac{k}{2\sigma}  \ \text{for}\ \frac{1}{2} \le \sigma \le 1, 
\end{cases}
\end{align*}
\begin{align} \label{eq:1.4}
\mathcal{G}_{\sigma,\nu}(t,\xi) :=
\begin{cases}
& 
\dfrac{e^{-\frac{1}{\nu}t |\xi|^{2(1-\sigma)} }}{\nu |\xi|^{2 \sigma} }
 \ \ \text{for} \ \ 0 < \sigma < \frac{1}{2}, \nu>0, \\
\ \\
& 
\dfrac{2 e^{-\frac{\nu}{2} t|\xi| } 
\sin \left( 
\frac{t |\xi|\sqrt{4-\nu^{2}} }{2}
\right)
}{|\xi|\sqrt{4-\nu^{2}}  }
 \ \ \text{for} \ \ \sigma =\frac{1}{2},\ 0<\nu<2, \\
\ \\
& 
t e^{-t|\xi|} 
\ \ \text{for} \ \ \sigma =\frac{1}{2},\ \nu=2, \\
\ \\
& 
\dfrac{2 e^{-\frac{\nu}{2}t|\xi|} \sinh \left( 
\frac{t |\xi| \sqrt{\nu^{2}-4} }{2}
\right)
}{|\xi| \sqrt{\nu^{2}-4} }
 \ \ \text{for} \ \ \sigma =\frac{1}{2},\ \nu>2,\\
\ \\
& 
\dfrac{ e^{-\frac{\nu}{2} t |\xi|^{2 \sigma} } 
\sin (t |\xi|)
}{|\xi|  }
\ \ \text{for} \ \ \frac{1}{2} < \sigma \le 1,\ \nu>0,
\end{cases}
\end{align}

\begin{align}  \label{eq:1.5}
\mathcal{H}_{\sigma,\nu}(t,\xi) :=
\begin{cases}
& 
e^{-\frac{1}{\nu}t |\xi|^{2(1-\sigma) }}
 \ \ \text{for} \ \ 0 < \sigma < \frac{1}{2}, \nu>0, \\
\ \\
& 
e^{-\frac{\nu}{2} t|\xi| } 
\cos \left( 
\frac{t |\xi|\sqrt{4-\nu^{2}} }{2}
\right)
+
\dfrac{\nu e^{-\frac{\nu}{2} t|\xi| } 
\sin \left( 
\frac{t |\xi|\sqrt{4-\nu^{2}} }{2}
\right)
}{\sqrt{4-\nu^{2}}  } \ \\
& \ \text{for} \ \ \sigma =\frac{1}{2},\ 0<\nu<2, \\
\ \\
&
(1+t |\xi|) e^{-t|\xi|} 
\ \ \text{for} \ \ \sigma =\frac{1}{2},\ \nu=2, \\
\ \\
& 
e^{-\frac{\nu}{2} t|\xi| } 
\cosh \left( 
\frac{t |\xi|\sqrt{\nu^{2}-4 }}{2}
\right) \\
& \ \ +
\dfrac{\nu e^{-\frac{\nu}{2} t|\xi| } 
\sinh \left( 
\frac{t |\xi|\sqrt{\nu^{2}-4} }{2}
\right)
}{\sqrt{\nu^{2}-4}  } \ \text{for} \ \ \sigma =\frac{1}{2},\ \nu>2, \\
\ \\
& 
e^{-\frac{\nu}{2} t |\xi|^{2 \sigma} } 
\cos (t |\xi|)
 \ \ \text{for} \ \ \frac{1}{2} < \sigma \le 1,\ \nu>0.
\end{cases}
\end{align}
We first mention the unique existence of the solution with decay properties
to problem \eqref{eq:1.1}. 
\begin{prop} \label{prop:1.1}
Let 
\begin{equation} \label{eq:1.6}
\left\{
\begin{split}
& \sigma \in \left(0, \frac{1}{2} \right), \quad  n \ge 2, \\ 
& \sigma = \frac{1}{2}, \quad  n \ge 1, \\ 
& \sigma \in \left(\frac{1}{2},1\right], \quad  n \ge 3, 
\end{split}
\right.
\end{equation}
$k_{0} \ge 0$ and $\nu >0$. 
Suppose that $(u_{0}, u_{1}) \in (H^{k_{0}+1} \cap L^{1}) \times  (H^{k_{0}} \cap L^{1})$.
Then, there exists a unique solution $u \in C([0,\infty); H^{k_{0}+1}) \cap C^{1}([0,\infty); H^{k_{0}}) $ to problem \eqref{eq:1.1} satisfying 
\begin{align} 
\label{eq:1.7}
&\| \partial_{t}^{\ell} \nabla^{k}_{x}u(t) \|_{2} \le C (1+t)^{-\gamma_{\sigma,k}-\ell}, 
\quad \sigma \in \left(0,\frac{1}{2} \right] \\
\label{eq:1.8}
& \| \partial_{t}^{\ell} \nabla^{k}_{x}u(t) \|_{2} \le C (1+t)^{-\gamma_{\sigma,k}-\frac{\ell}{2 \sigma}}, 
\quad \sigma \in \left(\frac{1}{2},1 \right]
\end{align}
for $\ell=0,1$ and $k \in [0, k_{0}+1]$, where $k+ \ell \le k_{0}+1$.
\end{prop}
Our next aim is to approximate the solution to \eqref{eq:1.1} by a constant multiple of the special functions
with a suitable lower bound.
We can now formulate our main results.
\begin{thm} \label{thm:1.2}
Under the same assumptions as in Proposition {\rm \ref{prop:1.1}},
it holds that  
\begin{align} \label{eq:1.9}
& \| \nabla_{x}^{k} (u(t) - m_{1} \mathcal{F}^{-1} [\mathcal{G}_{\sigma, \nu}(t)] ) \|_{2}
=o(t^{-\gamma_{\sigma,k}}), \ \text{for}\ \sigma \in (0, 1], \\
\label{eq:1.10}
& \| \partial_{t} \nabla^{k}_{x}(u(t) - m_{1}  \mathcal{F}^{-1} [\mathcal{G}_{\sigma, \nu}(t)] ) \|_{2}
=o(t^{-\gamma_{\sigma,k}-1})\ \ \text{for}\ \sigma \in \left(0, \frac{1}{2} \right], \\
& \label{eq:1.11}
\| \partial_{t} \nabla^{k}_{x}u(t) 
- \nabla^{k}_{x}m_{1}  \mathcal{F}^{-1} [e^{-\frac{\nu t |\xi|^{2 \sigma}}{2}} \cos (t |\xi|)]  \|_{2}
=o(t^{-\gamma_{\sigma,k}-\frac{1}{2 \sigma}})\ \ \text{for}\ \sigma \in \left(\frac{1}{2},1\right]
\end{align}
as $t \to \infty$,
where 
\begin{align} 
%\label{eq:1.12}
m_{1} := 
\int_{\R^{n}} u_{1}(y) dy.
\end{align}
Moreover there exists $C>0$ such that
\begin{align} \label{eq:1.13}
& C^{-1} t^{-\gamma_{\sigma,k}-\ell} \le \| \partial_{t}^{\ell} \nabla_{x}^{k} u(t) \|_{2} 
\le C t^{-\gamma_{\sigma,k}-\ell} \quad \sigma \in \left(0,\frac{1}{2} \right], \\
& C^{-1} t^{-\gamma_{\sigma,k}-\frac{\ell}{2 \sigma}} \le \| \partial_{t}^{\ell} \nabla_{x}^{k} u(t) \|_{2} 
\le C t^{-\gamma_{\sigma,k}-\frac{\ell}{2 \sigma}} \quad \sigma \in \left(\frac{1}{2},1 \right]  \label{eq:1.14}
\end{align}
for large $t$, where $\ell=0,1$, $k \in [0,k_{0}+1]$ and $k+ \ell \le k_{0}+1$.
\end{thm}
If $u_{1}(x)=0$,
we can assert the following series of approximation formulas of the solution to \eqref{eq:1.1}. 
\begin{thm} \label{thm:1.3}
Let $n \ge 1$, $\sigma \in (0, 1]$, $k_{0} \ge 0$ and 
$\nu >0$. 
If $u_{0} \in (H^{k_{0}+1} \cap L^{1})$ and $u_{1} \equiv 0$,
then it holds that  
\begin{align} \label{eq:1.15}
& \| \nabla_{x}^{k} (u(t) - m_{0} \mathcal{F}^{-1} [\mathcal{H}_{\sigma, \nu}(t)] ) \|_{2}
=o(t^{-\tilde{\gamma}_{\sigma,k}}), \ \text{for}\ \sigma \in (0, 1], \\
\label{eq:1.16}
& \| \partial_{t} \nabla^{k}_{x}(u(t) - m_{0}  \mathcal{F}^{-1} [\mathcal{H}_{\sigma, \nu}(t)] ) \|_{2}
=o(t^{-\tilde{\gamma}_{\sigma,k}-1})\ \ \text{for}\ \sigma \in \left(0, \frac{1}{2} \right], \\
& \label{eq:1.17}
\| \partial_{t} \nabla^{k}_{x}u(t) 
+\nabla^{k+1}_{x}m_{0}  \mathcal{F}^{-1} [e^{-\frac{\nu t |\xi|^{2 \sigma}}{2}} \sin (t |\xi|)]  \|_{2}
=o(t^{-\tilde{\gamma}_{\sigma,k}-\frac{1}{2 \sigma}})\ \ \text{for}\ \sigma \in \left(\frac{1}{2},1\right]
\end{align}
as $t \to \infty$,
where 
\begin{align*} 
%\label{eq:1.18}
m_{0} := 
\int_{\R^{n}} u_{0}(y) dy.
\end{align*}
Moreover, there exists $C>0$ such that
\begin{align*} 
%\label{eq:1.19}
& C^{-1} t^{-\tilde{\gamma}_{\sigma,k}-\ell} \le \| \nabla_{x}^{k} u(t) \|_{2} 
\le C t^{-\tilde{\gamma}_{\sigma,k}-\ell} \quad \sigma \in \left(0,\frac{1}{2} \right],\\
& C^{-1} t^{-\tilde{\gamma}_{\sigma,k}-\frac{\ell}{2 \sigma}} \le \| \partial_{t}^{\ell} \nabla_{x}^{k} u(t) \|_{2} 
\le C t^{-\tilde{\gamma}_{\sigma,k}-\frac{\ell}{2 \sigma}} \quad \sigma \in \left(\frac{1}{2},1 \right] 
% \label{eq:1.20}
\end{align*}
for large $t$, where $\ell=0,1$, $k \in [0,k_{0}+1]$ and $k+ \ell \le k_{0}+1$.
\end{thm}
%%%%%%
\begin{rem}
\eqref{eq:1.9} and \eqref{eq:1.15} state that the solution $u(t,x)$ of \eqref{eq:1.1} 
behaves like $m_{1} \mathcal{F}^{-1}[\G_{\sigma, \nu}(t)]$ 
and $m_{0} \mathcal{F}^{-1}[\mathcal{H}_{\sigma, \nu}(t)]$ for $t \to \infty$. 
The point of \eqref{eq:1.9} - \eqref{eq:1.11}
and \eqref{eq:1.15} - \eqref{eq:1.17} is that the asymptotic profile of $\partial_{t} \nabla^{k}_{x} u(t)$ 
for $\sigma \in (0, 1/2]$
is given by $m_{1} \partial_{t} \mathcal{F}^{-1}[\G_{\sigma, \nu}(t)]$ and
$m_{0} \partial_{t} \mathcal{F}^{-1}[\mathcal{H}_{\sigma, \nu}(t)]$, however, 
for $\sigma \in (1/2, 1]$,  
$\partial_{t} \nabla^{k}_{x} u(t)$ 
is not approximated by $m_{1} \partial_{t} \mathcal{F}^{-1}[\G_{\sigma, \nu}(t)]$ and
$m_{0} \partial_{t} \mathcal{F}^{-1}[\mathcal{H}_{\sigma, \nu}(t)]$ as  $t \to \infty$.
\end{rem}
%%%%%
Before closing this section, 
we summarize notation, which is used throughout this paper.
Let $\hat{f}$ denote the Fourier transform of $f$
defined by
\begin{align*}
\hat{f}(\xi) := c_{n}
\int_{\R^{n}} e^{-i x \cdot \xi} f(x) dx
\end{align*}
with $c_{n}= (2 \pi)^{-\frac{n}{2}}$.
Also, let $\mathcal{F}^{-1}[f]$ or $\check{f}$ denote the inverse
Fourier transform.

For $k \ge 0$, let  $H^{k}(\R^{n})$ be the Sobolev space;
\begin{equation*}
H^{k}(\R^{n})
  :=\Big\{ f:\R^{n} \to \R;
        \| f \|_{H^{k}(\R^{n})} 
        := (\Vert f \Vert_{2}^{2} + 
         \Vert \nabla_{x}^{k} f \Vert_{2}^{2})^{1/2}< \infty 
     \Big\},
\end{equation*}
where $L^{p}(\R^{n})$ is the usual Lebesgue space for $1 \le p \le \infty$.
For the notation of function spaces, 
the domain $\R^{n}$ is often abbreviated, and we frequently use the notation $\| f \|_{p} =\| f \|_{L^{p}(\R^{n})}$ without confusion.
Furthermore, in the following, $C$ denotes a positive constant, which may change from line to line.

This paper is organized as follows.
In section 2, 
we set up notation of the solution formula by the Fourier multiplier expression, 
which is useful to describe the asymptotic profiles of solutions.
Section 3 describes several results of \cite{DR} in terms of our notation.
Section 4 is devoted to the study of the behaviors of the Fourier multipliers in the Fourier space.
In section 5, we prove the upper bound of the norms of the evolution operators, which mean decay properties. 
Section 6 provides approximation formulas of the evolution operators of \eqref{eq:1.1}.
In section 7, we prove our main results. 
%

%%%%%%%%%%%%%%%%%%%%%
%%%%%%%%%%%%%%%%%%%%%%%%
%%%%%%%%%%%%%%%%%%%%%%%%%
\section{Solution formula}
In this section, we formulate the solution of \eqref{eq:1.1} by using the Fourier multiplier theory.
We remark that our new ingredient here is the case for $\sigma \in (1/2, 1]$
and $\sigma = \frac{1}{2}$ with $\nu \neq 2, >0$.
It is useful to obtain the asymptotic profile of solutions.  
The results in this section is essentially obtained by D'Abicco-Reissig \cite{DR}, however, for the reader's convenience, we repeat the derivation of the evolution operators to 
\eqref{eq:1.1}. 

We begin with recalling the characteristic roots of \eqref{eq:1.1}.
Applying the Fourier transform to the equation \eqref{eq:1.1}, 
we see 
\begin{equation} \label{eq:2.1}
\left\{
\begin{split}
& \partial_{t}^{2} \hat{u} +|\xi|^{2} \hat{u} 
+ \nu |\xi|^{2 \sigma} \partial_{t} \hat{u} =0, \quad t>0, \quad x \in \R^{n}, \\
& \hat{u}(0,\xi)=\hat{u}_{0}(\xi), \quad \partial_{t} \hat{u}(0,x)=\hat{u}_{1}(\xi) , \quad x \in \R^{n}, 
\end{split}
\right.
\end{equation}
and we have the characteristic equations  
$\lambda^{2} + \nu |\xi|^{2 \sigma} \lambda + |\xi|^{2} =0$.
Then we see that the characteristic roots
$\lambda_{\pm}$
are given by 
\begin{align*} 
\lambda_{\pm}:=
-\frac{\nu |\xi|^{2 \sigma}}{2}
\pm
\sqrt{
\frac{
\nu^{2}
}{4}
 |\xi|^{4 \sigma}
-
|\xi|^{2} 
},
\end{align*}
and roughly speaking, 
for small $|\xi|$,
their behaviors are given by  
\begin{equation*} 
\begin{split}
\lambda_{+} =
\begin{cases}
& 
\frac{-2 |\xi|^{2(1-\sigma)}}
{
\nu(1+ \sqrt{
1-\frac{4 |\xi|^{2-4 \sigma}}{\nu^{2} }
})
} 
\sim -\frac{1}{\nu} |\xi|^{2(1-\sigma) }
 \ \ \text{for} \ \ 0 < \sigma < \frac{1}{2}, \nu>0, \\
& \frac{\nu |\xi|}{2} \pm \frac{\sqrt{4-\nu^{2}} |\xi| i}{2} \qquad (0<\nu<2), \\
& |\xi| \qquad (\nu=2), \quad (\text{multiplicity}\ 2), \\
& \frac{\nu |\xi|}{2} \pm \frac{\sqrt{\nu^{2} -4} |\xi| }{2} \qquad (\nu>2), \\
& -\frac{\nu}{2}|\xi|^{2 \sigma} +i |\xi|
 \ \ \text{for} \ \ \frac{1}{2} < \sigma \le 1,\ \nu>0, 
\end{cases}
\end{split}
\end{equation*}
and 
\begin{equation*} 
\begin{split}
\lambda_{-} =
\begin{cases}
& 
\frac{-\nu |\xi|^{2 \sigma} - \nu |\xi|^{2 \sigma} \sqrt{1-\frac{4|\xi|^{2-4\sigma}}{\nu^{2}}}}{2}
\sim -\nu |\xi|^{2\sigma}
 \ \ \text{for} \ \ 0 < \sigma < \frac{1}{2}, \nu>0, \\
& \frac{\nu |\xi|}{2} \pm \frac{\sqrt{4-\nu^{2}} |\xi| i}{2} \qquad (0<\nu<2), \\
& |\xi| \qquad (\nu=2), \quad (\text{multiplicity}\ 2), \\
& \frac{\nu |\xi|}{2} \pm \frac{\sqrt{\nu^{2} -4} |\xi| }{2} \qquad (\nu>2), \\
& -\frac{\nu}{2}|\xi|^{2 \sigma} +i |\xi|
 \ \ \text{for} \ \ \frac{1}{2} < \sigma \le 1,\ \nu>0.
\end{cases}
\end{split}
\end{equation*}
%
%%%%%% cut off %%%%%%
Thereafter 
we introduce radial cut-off functions 
which will be used in the proofs 
to aligned to the low-, middle- and high-frequency parts. 
Let $\chi_L$, $\chi_M $ and $\chi_H  \in C^{\infty}(\R^{n})$ be 
\begin{gather*}
\chi_L (\xi) = \begin{cases}
	1, \quad &|\xi| \leq \frac{\rho}{2}, \\
	0, \quad &|\xi| \geq \rho, 
	\end{cases} \qquad
\chi_H (\xi) = \begin{cases}
	1, \quad &|\xi| \geq 2, \\
	0, \quad &|\xi| \leq 4, 
	\end{cases} \\ 
\chi_M (\xi) = 1- \chi_L (\xi) - \chi_H (\xi). 
\end{gather*}
Here we choose $\rho>0$ satisfying 
\begin{align} \label{eq:2.2}
\rho< 
\begin{cases}
& \frac{1}{2} \left( \frac{\nu}{2} \right)^{\frac{1}{1-2 \sigma}}
\ \text{for}\ \sigma \in (0, \frac{1}{2}), \\
&  \frac{1}{2} 
\ \text{for}\ \sigma=\frac{1}{2}, \\
& \frac{1}{2} \left( \frac{2}{\nu} \right)^{\frac{1}{2 \sigma-1}}
\ \text{for}\ \sigma \in (\frac{1}{2},1].
\end{cases}
\end{align}
%%%%%%%%%%%%%%%%%%%%%%%%%%%%%%%%%%%%%%%%%%%%%%%%%%%%%%%%%%%%%%%%%%%%%%%%%%%%%%%%%%%%%%%%%%%%%%%%%%%%%%%%%%%% 

\subsection{The case for $\sigma \in (0, 1/2)$.}
When $\sigma \in (0, \frac{1}{2})$, 
we can write the solution of \eqref{eq:2.1} by using constants $C_{1}$ and $C_{2}$ such as 
\begin{equation*}
\hat{u}(t) = C_{1} e^{\lambda_{+}t}+ C_{2} e^{\lambda_{-} t}.
\end{equation*}
The direct calculation implies 
\begin{equation*}
C_{1} = 
\frac{
-\lambda_{-} \hat{u}_{0}+ \hat{u}_{1}}{\lambda_{+} - \lambda_{-}}, 
\quad
C_{2} = \frac{
\lambda_{+} \hat{u}_{0} + \hat{u}_{1}
}{\lambda_{+} - \lambda_{-}}, 
\end{equation*}
where
\begin{equation*}
\begin{split}
\lambda_{+} - \lambda_{-}
= \sqrt{\nu^{2} |\xi|^{4 \sigma} - 4|\xi|^{2}}
=
\begin{cases}
& \nu |\xi|^{2 \sigma} \sqrt{1-\frac{4 |\xi|^{2- 4 \sigma}}{\nu^{2}}}
\sim \nu |\xi|^{2 \sigma}, |\xi| \to 0, \\
& 2 i |\xi| \sqrt{1-\frac{\nu^{2}|\xi|^{4 \sigma-2}}{4}}
\sim 2i |\xi|, |\xi| \to \infty.
\end{cases}
\end{split}
\end{equation*}
Therefore we obtain the following Fourier multiplier expression of the solution $u(t,x)$:
\begin{equation} \label{eq:2.3}
u(t) = J_{1}(t)u_{0}+ J_{2}(t)u_{1} +J_{3}(t)u_{0} +J_{4}(t) u_{1},
\end{equation}
where 
\begin{equation} \label{eq:2.4}
\begin{split}
& J_{1}(t)u_{0}
:= 
\mathcal{F}^{-1}
\left[
\frac{
-\lambda_{-} e^{\lambda_{+}t}
}{\lambda_{+} - \lambda_{-}}
\hat{u}_{0}
\right], \quad  
J_{2}(t)u_{1} :=
\mathcal{F}^{-1}
\left[
\frac{
e^{\lambda_{+}t}
}{\lambda_{+} - \lambda_{-}}
\hat{u}_{1}
\right], \\
& J_{3}(t)u_{0}
:=
\mathcal{F}^{-1}
\left[
\frac{
\lambda_{+}  e^{\lambda_{-}t}
}{\lambda_{+} - \lambda_{-}}
\hat{u}_{0}
\right], \quad 
J_{4}(t)u_{1}
:=
\mathcal{F}^{-1}
\left[
\frac{
e^{\lambda_{-}t}
}{\lambda_{+} - \lambda_{-}}
\hat{u}_{1}
\right].
\end{split}
\end{equation}
%%%%%%%%%%%%%%%%%%%%
%%%%%%%%%%%%%%%%%%%
%%%%%%%%%%%% decomposition by the frequency %%%%%%%
By using the cut-off functions $\chi_{k}$ ($k=L,M,H$), 
we also have the localized operators $J_{jk}(t)g$ $(j=1,2,3,k=L,M,H)$ defined by  
\begin{equation} \label{eq:2.5}
\begin{split}
& J_{jk}(t)g
:= 
\mathcal{F}^{-1}
\left[
\J_{jk}(t, \xi)
\chi_{k}\hat{g}
\right],
\end{split}
\end{equation}
where we denote 
\begin{equation} \label{eq:2.6}
\begin{split}
& \J_{1k}(t, \xi) := \frac{
-\lambda_{-} e^{\lambda_{+}t}
}{\lambda_{+} - \lambda_{-}} \chi_{k},\quad
\J_{2k}(t, \xi) := \frac{
e^{\lambda_{+}t}
}{\lambda_{+} - \lambda_{-}} \chi_{k},\\
& \J_{3k}(t,\xi) 
:=
\frac{ \lambda_{+}
e^{\lambda_{-}t}
}{\lambda_{+} - \lambda_{-}}
\chi_{k}, \quad 
\J_{4k}(t,\xi) 
:=
\frac{
e^{\lambda_{-}t}
}{\lambda_{+} - \lambda_{-}}
\chi_{k}.
\end{split}
\end{equation}
%
%%%%%%%%%%%%%%%%%%%%%%%%%%%
%%%%%%%%%%%%%%%%%%%%%%%%%%%%%
\subsection{The case for $\sigma \in [1/2,1]$.}
%%%%%%%%%%%%%%%%%%%%%%
%%%%%%%%%%%%%%%%%%%
%%%%%%%%%%%%%%%%%%%%%%
For the case $\sigma \in (1/2,1]$, 
we can choose constants $C_{1}$ and $C_{2}$ such as  
\begin{equation*}
\hat{u}(t) = C_{1} e^{-\frac{\nu|\xi|^{2 \sigma} t}{2}} \cos (t |\xi| \phi_{\sigma})
+ C_{2} e^{-\frac{\nu|\xi|^{2 \sigma} t}{2}} \sin (t |\xi| \phi_{\sigma}),
\end{equation*}
where 
\begin{align} \label{eq:2.7}
\phi_{\sigma} = \phi_{\sigma}(\xi) = \sqrt{1-\frac{\nu^{2} |\xi|^{4 \sigma-2}}{4}}, 
\end{align}
and this leads to
\begin{equation*}
C_{1}=\hat{u}_{0}, \quad 
C_{2}= 
\frac{\nu |\xi|}{2 \phi_{\sigma}} \hat{u}_{0} + \frac{1}{|\xi| \phi_{\sigma}} \hat{u}_{1}.
\end{equation*}
Namely, we find 
\begin{equation} \label{eq:2.8}
u(t) = K_{1}(t)u_{0}+ K_{2}(t)u_{0} +K_{3}(t)u_{1},
\end{equation}
where 
\begin{equation} \label{eq:2.9}
\begin{split}
& K_{1}(t)g
:= 
\mathcal{F}^{-1}
\left[
e^{-\frac{\nu|\xi|^{2 \sigma} t}{2}} \cos (t |\xi| \phi_{\sigma})
\hat{g}
\right], \\  
& K_{2}(t)g
:=
\mathcal{F}^{-1}
\left[
\frac{
e^{-\frac{\nu|\xi|^{2 \sigma} t}{2}} \nu |\xi| \sin (t |\xi| \phi_{\sigma})
}{2 \phi_{\sigma}}
\hat{g}
\right], \\
& K_{3}(t)g
:=
\mathcal{F}^{-1}
\left[
\frac{
e^{-\frac{\nu|\xi|^{2 \sigma} t}{2}} \sin (t |\xi| \phi_{\sigma})
}{|\xi|\phi_{\sigma}}
\hat{g}
\right].
\end{split}
\end{equation}
We also introduce the localized operators $K_{jk}(t)$ $(j=1,2,3,\ k=L,M,H)$ of $K_{j}(t)$ ($j = 1,2,3$) as follows:
\begin{equation} \label{eq:2.10}
\begin{split}
K_{jk}(t)g
:=
\mathcal{F}^{-1}
\left[
\K_{jk}(t,\xi)
\chi_{j} \hat{g}
\right],
\end{split}
\end{equation}
where $\K_{jk}(t,\xi)$ is defined by 
\begin{equation} \label{eq:2.11}
\begin{split}
& \K_{1k}(t, \xi )
:= e^{-\frac{\nu|\xi|^{2 \sigma} t}{2}} \cos (t |\xi| \phi_{\sigma}) \chi_{k}, \quad
\K_{2k}(t, \xi )
:= 
\frac{
e^{-\frac{\nu|\xi|^{2 \sigma} t}{2}} \nu |\xi|\sin (t |\xi| \phi_{\sigma})
}{2\phi_{\sigma}} \chi_{k}, \\
& \K_{3k}(t, \xi )
:= 
\frac{
e^{-\frac{\nu|\xi|^{2 \sigma} t}{2}} \sin (t |\xi| \phi_{\sigma})
}{|\xi|\phi_{\sigma}} \chi_{k}.
\end{split}
\end{equation}
We continue, in a similar fashion, to obtain the expression of the solution with $\sigma=\frac{1}{2}$ 
corresponding to the value of $\nu \neq 2$. 
Namely, we have 
\begin{equation*} 
%\label{eq:2.12}
\begin{split}
\hat{u}(t) & = 
e^{-\frac{\nu |\xi| t}{2}}
\cosh \left( 
\frac{t |\xi| \sqrt{\nu^{2} -4} }{2}
\right) \hat{u}_{0}
+ \frac{
e^{-\frac{\nu |\xi| t}{2}}
\nu
}
{
\sqrt{\nu^{2} -4}
}
\sinh \left( 
\frac{t |\xi| \sqrt{\nu^{2} -4} }{2}
\right) \hat{u}_{0} \\
& + \frac{
2 e^{-\frac{\nu |\xi| t}{2}}
}
{ |\xi|
\sqrt{\nu^{2} -4}
}
\sinh \left( 
\frac{t |\xi| \sqrt{\nu^{2} -4} }{2}
\right) \hat{u}_{1} 
\end{split}
\end{equation*}
for $\nu>2$ and 
\begin{equation*} 
%\label{eq:2.13}
\begin{split}
\hat{u}(t) & = 
e^{-\frac{\nu |\xi| t}{2}}
\cos \left( 
\frac{t |\xi| \sqrt{4-\nu^{2} } }{2}
\right) \hat{u}_{0}
+ \frac{
e^{-\frac{\nu |\xi| t}{2}}
\nu
}
{
\sqrt{\nu^{2} -4}
}
\sin \left( 
\frac{t |\xi| \sqrt{4-\nu^{2} } }{2}
\right) \hat{u}_{0} \\
& + \frac{
2 e^{-\frac{\nu |\xi| t}{2}}
}
{ |\xi|
\sqrt{4-\nu^{2}}
}
\sin \left( 
\frac{t |\xi| \sqrt{4-\nu^{2}} }{2}
\right) \hat{u}_{1} 
\end{split}
\end{equation*}
for $0< \nu< 2$.
For simplicity we introduce the notation 
\begin{equation} \label{eq:2.14}
\begin{split}
\tilde{\J}_{1}(t, \xi) & := 
e^{-\frac{\nu |\xi| t}{2}}
\cosh \left( 
\frac{t |\xi| \sqrt{\nu^{2} -4} }{2}
\right), \quad 
\tilde{\J}_{2}(t, \xi) :=  \frac{
e^{-\frac{\nu |\xi| t}{2}}
\nu
}
{
\sqrt{\nu^{2} -4}
}
\sinh \left( 
\frac{t |\xi| \sqrt{\nu^{2} -4} }{2}
\right) \\
\tilde{\J}_{3}(t, \xi) & :=  \frac{
2 e^{-\frac{\nu |\xi| t}{2}}
}
{ |\xi|
\sqrt{\nu^{2} -4}
}
\sinh \left( 
\frac{t |\xi| \sqrt{\nu^{2} -4} }{2}
\right), 
\end{split}
\end{equation}
\begin{equation} \label{eq:2.15}
\begin{split}
\tilde{\K}_{1}(t, \xi) & := 
e^{-\frac{\nu |\xi| t}{2}}
\cos \left( 
\frac{t |\xi| \sqrt{4-\nu^{2} } }{2}
\right), \quad
\tilde{\K}_{2}(t, \xi)  :=  \frac{
e^{-\frac{\nu |\xi| t}{2}}
\nu 
}
{
\sqrt{\nu^{2} -4}
}
\sin \left( 
\frac{t |\xi| \sqrt{4-\nu^{2} } }{2}
\right), \\
\tilde{\K}_{3}(t, \xi) & := \frac{
2 e^{-\frac{\nu |\xi| t}{2}}
}
{ |\xi|
\sqrt{4-\nu^{2}}
}
\sin \left( 
\frac{t |\xi| \sqrt{4-\nu^{2}} }{2}
\right),
\end{split}
\end{equation}
and 
\begin{align} \label{eq:2.16}
\tilde{J}_{j}(t) g :=
\mathcal{F}^{-1}[\tilde{\J}_{j}(t,\xi) \hat{g}], \quad 
\tilde{K}_{j}(t) g :=
\mathcal{F}^{-1}[\tilde{\K}_{j}(t,\xi) \hat{g}]
\end{align}
for $j=1,2,3$.
For the case $\nu=2$, 
as was pointed out in several previous results
(see e.g. \cite{D}, \cite{DR} and \cite{NR}),
we can obtain
\begin{align*} 
%\label{eq:2.17}
\hat{u}(t)
= (e^{-t|\xi|} +t e^{-t|\xi|}|\xi| )\hat{u}_{0}
+ t e^{-t|\xi|} \hat{u}_{1}
\end{align*}
and then we define 
\begin{align} \label{eq:2.18}
\E_{1}(t,\xi) :=e^{-t|\xi|}, \quad \E_{2}(t,\xi):=t e^{-t|\xi|}|\xi|, 
\quad 
\E_{3}(t,\xi):= t e^{-t|\xi|}
\end{align}
and 
\begin{align} \label{eq:2.19}
E_{j}(t) g:=\mathcal{F}^{-1} [\E_{j}(t,\xi) \hat{g}]
\end{align}
for $j=1,2,3$.

Therefore, we have just arrived at the expression of the solution with $\sigma = \frac{1}{2}$ for \eqref{eq:1.1} by 
\begin{align} \label{eq:2.20}
u(t) = 
\begin{cases}
& \tilde{K}_{1}(t) u_{0}+\tilde{K}_{2}(t) u_{0} +\tilde{K}_{3}(t) u_{1}\ \text{for} \ 0<\nu<2, \\
&  E_{1}(t) u_{0} +E_{2}(t) u_{0} + E_{3}(t)u_{1}\ \text{for} \ \nu=2, \\
&  \tilde{J}_{1}(t) u_{0}+\tilde{J}_{2}(t) u_{0} +\tilde{J}_{3}(t) u_{1}\ \text{for} \ \nu>2.
\end{cases}
\end{align}
\begin{rem}
We note that the choice of $\rho$ defined by \eqref{eq:2.2} means that 
the positive root of $\tau^{4 \sigma-2} = \frac{4}{\nu^{2}}$ does not belong to $\supp \chi_{L}$
for $\sigma \in (0,1] \setminus \{ \frac{1}{2}\}$.
\end{rem}
%
%%%%%%%%%%%%%%%%%%%%%%%%%%%%%%
%%%%%%%%%%%%%%%%%%%%%%%%%%%%%%%%%
\section{Restatement of the results by \cite{DR}}
%%%%%%%%%%%%%%%%%%%%%%%%%%%%%%%%%
%%%%%%%%%%%%%%%%%%%%%%%%%%%%%%%%%%
%%%%%%%%%%%%%%%%%%%%%%%%%%%%%%%%%%%%%%%
Our results here are closely related to those of \cite{DR}.
In this section, we summarize, without proofs, the precise statements of their results,
the point-wise estimates of the fundamental solutions for \eqref{eq:1.1} in the Fourier space, and decay estimates of the solution for \eqref{eq:1.1} by using our notation and terminology introduced in the previous section. 
The following lemmas show the behavior of $\J_{jk}(t, \xi)$ for $j=1,2,3,4$ and $k=L,M,H$ in the Fourier
space.
\begin{lem} \label{lem:3.1}
Let $n \ge 1$, $\sigma \in (0, \frac{1}{2})$, $k \ge 0$ and $\ell=0,1$.
Then, there exist $C>0$ and $c>0$ such that 
\begin{equation} \label{eq:3.1}
\begin{split}
& 
|\xi|^{k}
|
\partial_{t}^{\ell}
\J_{1L}(t,\xi)
|
\le C e^{
-c(1+t)|\xi|^{2(1-\sigma)}
}
|\xi|^{2(1-\sigma) \ell +k}  \chi_{L},\\
& 
|\xi|^{k}
|
\partial_{t}^{\ell}
\J_{2L}(t,\xi)
| 
\le C e^{
-c (1+t)|\xi|^{2(1-\sigma)}
}
|\xi|^{2(1-\sigma) \ell-2\sigma +k} \chi_{L},\\
& 
|\xi|^{k}
|
\partial_{t}^{\ell}
\J_{3L}(t,\xi)
| 
\le C e^{
-c(1+t)|\xi|^{2\sigma}
}
|\xi|^{2\ell \sigma+2(1-2\sigma)+k } \chi_{L}, \\
&
|\xi|^{k}
|
\partial_{t}^{\ell}
\J_{4L}(t,\xi)
| 
\le C e^{
-c(1+t)|\xi|^{2\sigma}
}
|\xi|^{2 \ell \sigma -2\sigma+ k} \chi_{L},
\end{split}
\end{equation}
where $\J_{jL}(t,\xi)$ $(j=1,2,3,4)$ are defined by \eqref{eq:2.6}.
\end{lem}
\begin{lem} \label{lem:3.2}
Let $n \ge 1$, $\sigma \in (0, \frac{1}{2})$, $k \ge 0$ and $\ell=0,1$.
Then, there exist $C>0$ and $c>0$ such that 
\begin{equation*} 
%\label{eq:3.2}
\begin{split}
& 
|\xi|^{k}
\sum_{j=1,3}
(
|
\partial_{t}^{\ell} \J_{jM}(t,\xi)
|
+
|
\partial_{t}^{\ell} \J_{jH}(t,\xi)
|
)
\le C e^{
-ct|\xi|^{2\sigma}
}
|\xi|^{\ell+k}  (\chi_{M} +\chi_{H}),\\
& 
|\xi|^{k}
\sum_{j=2,4}
(
|
\partial_{t}^{\ell} \J_{jM}(t,\xi)
|
+
|
\partial_{t}^{\ell} \J_{jH}(t,\xi)
|
)
\le C e^{
-ct|\xi|^{2\sigma}
}
|\xi|^{\ell-1+k}  (\chi_{M} +\chi_{H}),
\end{split}
\end{equation*}
where $\J_{jk}(t,\xi)$ $(j=1,2,3,4)$, $(k=M,H)$ are defined by \eqref{eq:2.6}.
\end{lem}
The behavior of $\K_{jk}(t, \xi)$ for $j=1,2,3$ and $k=L,M,H$ is estimated as follows.
\begin{lem} \label{lem:3.3}
Let $n \ge 1$, $\sigma \in (\frac{1}{2},1]$, $\ell=0,1$ and $k \ge 0$.
Then, there exist $C>0$ and $c>0$ such that 
\begin{equation} \label{eq:3.3}
\begin{split}
& 
|\xi|^{k}
|
\partial_{t}^{\ell} \K_{1L}(t,\xi)
|
\le C e^{-c(1+t) |\xi|^{2 \sigma}}
|\xi|^{\ell+k}  \chi_{L},\\
& 
|\xi|^{k}
|
\partial_{t}^{\ell} 
\K_{2L}(t,\xi)
| 
=
|\xi|^{k}
\left|\frac{\nu}{2} \Delta
\partial_{t}^{\ell} 
\K_{3L}(t,\xi)
\right| 
\le C e^{-c (1+t) |\xi|^{2 \sigma}}
|\xi|^{\ell+k+1}  \chi_{L},\\
& |\xi|^{k}
|
\partial_{t}^{\ell} 
\K_{3L}(t,\xi)
| 
\le C e^{-c (1+t) |\xi|^{2 \sigma}}
|\xi|^{\ell+k-1}  \chi_{L},\\
\end{split}
\end{equation}
where $\K_{jL}(t,\xi)$ $(j=1,2,3)$ are defined by \eqref{eq:2.11}.
\end{lem}
\begin{lem} \label{lem:3.4}
Let $n \ge 1$, $\sigma \in (\frac{1}{2},1)$, $k \ge 0$ and $\ell=0,1$.
Then, there exist $C>0$ and $c>0$ such that 
\begin{align*} 
%\label{eq:3.4}
%
& 
|\xi|^{k}
\sum_{k=1,2}
(
|
\partial_{t}^{\ell} 
\K_{1M}(t,\xi)
|
+
|
\partial_{t}^{\ell} 
\K_{1H}(t,\xi)
|
) 
\le C
|\xi|^{k+2(1-\sigma) \ell } e^{
-ct|\xi|^{2(1-\sigma)}
} (\chi_{M} +\chi_{H}), \\
& 
|\xi|^{k}
(
|
\partial_{t}^{\ell} 
\K_{3M}(t,\xi)
|
+
|
\partial_{t}^{\ell} 
\K_{3H}(t,\xi)
|
) 
\le C|\xi|^{k-2 \sigma+2 \sigma \ell}
e^{
-ct|\xi|^{2(1-\sigma)}
}
(\chi_{M} +\chi_{H}), 
%\label{eq:3.5}
%
\end{align*}
where $\K_{jk}(t,\xi)$ $(j=1,2,3, k=M,H)$ are defined by \eqref{eq:2.11}.
\end{lem}

%%%%
%%%

%%%%%%%%%%%%%%%%%%%%%%%%%%%%%%%%%%%%%%%%
\section{Point-wise estimates in the Fourier space}
This section deals with point-wise estimates of the Fourier multipliers in the Fourier space.
The results here play crucial roles to show our main results.
\subsection{The case for $\sigma \in (0, 1/2)$.}
This subsection is devoted to the estimates for $\J_{1L}(t,\xi)
-
e^{
-\frac{t}{\nu}
|\xi|^{2(1-\sigma)}
}
\chi_{L}$ and 
$
\J_{2L}(t,\xi)
-
\frac{
e^{
-\frac{t}{\nu}
|\xi|^{2(1-\sigma)
}
}
}{ \nu |\xi|^{2 \sigma}}\chi_{L} 
$.
In other words, 
the following lemmas mean that 
$\J_{1L}(t,\xi)$ and $\J_{2L}(t,\xi)$
behave like $e^{
-\frac{t}{\nu}
|\xi|^{2(1-\sigma)}
}
\chi_{L}$
and 
$\frac{
e^{
-\frac{t}{\nu}
|\xi|^{2(1-\sigma)
}
}
}{ \nu |\xi|^{2 \sigma}}\chi_{L} $, 
respectively.
\begin{lem} \label{lem:4.1}
Let $n \ge 1$, $\sigma \in (0, \frac{1}{2})$ and $k \ge 0$.
Then, there exist $C>0$ and $c>0$ such that 
\begin{align} \label{eq:4.1}
& 
|\xi|^{k}
\left|
\J_{1L}(t,\xi)
-
e^{
-\frac{t}{\nu}
|\xi|^{2(1-\sigma)}
}
\chi_{L} \right| 
\le C e^{
-c(1+t)
|\xi|^{2(1-\sigma)}
}
|\xi|^{k} 
(t |\xi|^{2(2-3\sigma)} + |\xi|^{2(1-2\sigma)}) \chi_{L},\\
& 
|\xi|^{k}
\left|
\J_{2L}(t,\xi)
-
\frac{
e^{
-\frac{t}{\nu}
|\xi|^{2(1-\sigma)
}
}
}{ \nu |\xi|^{2 \sigma}}\chi_{L} \right| 
\le C e^{
-c(1+t)
|\xi|^{2(1-\sigma)}
}
|\xi|^{k} 
(t |\xi|^{2(2-3\sigma)-2 \sigma} + |\xi|^{2(1-2\sigma)- 2 \sigma}) \chi_{L},
\label{eq:4.2}
\end{align}
where $\J_{1L}(t,\xi)$ and $\J_{2L}(t,\xi)$ are defined by \eqref{eq:2.6}.
\end{lem}
\begin{proof}
At first, we show \eqref{eq:4.1}.
Noting that 
\begin{equation*} 
%\label{eq:4.3}
\begin{split}
 \lambda_{+}+\frac{|\xi|^{2(1-\sigma)}}{\nu}
& =
\frac{|\xi|^{2(1-\sigma)}}{\nu}
\left( 
-\frac{2}{1+\sqrt{1-\frac{4|\xi|^{2 -4 \sigma}}{\nu^{2}}}}
+1
\right) \\
& =
\frac{-4 |\xi|^{2(2-3\sigma)} }
{ \nu^{3}
\left(
1+\sqrt{
1-\frac{4|\xi|^{2 -4 \sigma}}{\nu^{2}}
}
\right)^{2} 
} \le 0,
\end{split}
\end{equation*}
we see
\begin{align} \label{eq:4.4}
\left| 
\lambda_{+}+\frac{1}{\nu}
|\xi|^{2(1-\sigma)}
\right| \le C |\xi|^{2(2-3 \sigma)}
\end{align}
for small $|\xi|$.
On the other hand, 
the mean value theorem yields
\begin{align*}
e^{t \lambda_{+}+\frac{t}{\nu}
|\xi|^{2(1-\sigma)} }
-1
=
\left(
t \lambda_{+}+\frac{t}{\nu}
|\xi|^{2(1-\sigma)}
\right)
e^{\theta(t \lambda_{+}+\frac{t}{\nu}
|\xi|^{2(1-\sigma)} )}
\end{align*}
for some $\theta \in (0,1)$, and so we have
\begin{align} \label{eq:4.5}
\left|
e^{t \lambda_{+}+\frac{t}{\nu}
|\xi|^{2(1-\sigma)} }
-1
\right|
\le C t 
 |\xi|^{2(2-3 \sigma)}
\end{align}
%%%%%%%
by \eqref{eq:4.4}.
%%%%%
Moreover, 
the direct computation gives  
\begin{equation*} 
\begin{split}
\left(
\frac{-\lambda_{-}}{\lambda_{+} - \lambda_{-}} -1
\right)\chi_{L}
& = 
\left(
\frac{ 
1+\sqrt{
1-\frac{4}{\nu^{2}}
|\xi|^{2(1-2\sigma)}
}
}{ 
\sqrt{
1-\frac{4}{\nu^{2}}
|\xi|^{2(1-2\sigma)}
}
}
-1
\right)\chi_{L} \\
& = 
\frac{ 
-4
|\xi|^{2(1-2\sigma)}
}{ \nu^{2} \left(1+
\sqrt{
1-\frac{4}{\nu^{2}}
|\xi|^{2(1-2\sigma)}
}
\right)
}\chi_{L}
\end{split}
\end{equation*}
and 
\begin{equation} \label{eq:4.6} 
\begin{split}
\left|
\left(
\frac{-\lambda_{-}}{\lambda_{+} - \lambda_{-}} -1
\right)\chi_{L} \right|
& \le C|\xi|^{2(1-2\sigma)}\chi_{L}
\end{split}
\end{equation}
for small $|\xi|$.
Combining \eqref{eq:4.5} and \eqref{eq:4.6}, 
we arrive at the estimate 
\begin{align*}
& |\xi|^{k}
\left|
\J_{1L}(t,\xi)
-
e^{
-\frac{t}{\nu}
|\xi|^{2(1-\sigma)}
}
\chi_{L} \right| \\
& \le 
|\xi|^{k} e^{
-\frac{t}{\nu}
|\xi|^{2(1-\sigma)}
} \chi_{L}
\left(
\left|
\frac{-\lambda_{-}
(e^{
\lambda_{+}t +\frac{t}{\nu}
|\xi|^{2(1-\sigma)}
} -1)}
{\lambda_{+} -\lambda_{-}} \right| 
+
\left|
\left(
\frac{-\lambda_{-}}
{\lambda_{+} -\lambda_{-}}
-1 \right) \right| 
\right)
\\
& \le C e^{
-\frac{(1+t)}{\nu}
|\xi|^{2(1-\sigma)}
}
|\xi|^{k} 
(t |\xi|^{2(2-3\sigma)} + |\xi|^{2(1-2\sigma)}) \chi_{L},
\end{align*}
which is the desired estimate \eqref{eq:4.1}.
Next, we show \eqref{eq:4.2}.
It is easy to see that
\begin{equation} \label{eq:4.7}
\begin{split}
& \left|
 \frac{1}{
\sqrt{
1-\frac{4}{\nu^{2}}
|\xi|^{2(1-2\sigma)}
}
}
-1
 \right|
 = 
\left|
 \frac{1-\sqrt{
1-\frac{4}{\nu^{2}}
|\xi|^{2(1-2\sigma)}
}}{
\sqrt{
1-\frac{4}{\nu^{2}}
|\xi|^{2(1-2\sigma)}
}
}
\right| \\
& 
 = 
\left|
 \frac{
|\xi|^{2(1-2\sigma)}
}{\frac{4}{\nu^{2}}
\sqrt{
1-\frac{4}{\nu^{2}}
|\xi|^{2(1-2\sigma)}
}
\left(
1+\sqrt{
1-\frac{4}{\nu^{2}}
|\xi|^{2(1-2\sigma)}
}
\right)
}
\right| \le C |\xi|^{2(1-2\sigma)}
\end{split}
\end{equation}
for small $|\xi|$.
Thus by \eqref{eq:4.5} and \eqref{eq:4.7}, 
we can obtain 
\begin{equation*}
\begin{split}
& |\xi|^{k} \left|
\J_{2L}(t,\xi)
-
\frac{
e^{
-\frac{t}{\nu}
|\xi|^{2(1-\sigma)
}
}
}{ \nu |\xi|^{2 \sigma}} \right| \\
& \le \frac{
e^{
-\frac{t}{\nu}
|\xi|^{2(1-\sigma)}
}
|\xi|^{k}
\chi_{L}
}{ \nu |\xi|^{2 \sigma}}
\left(
\left|
 \frac{
e^{t \lambda_{+}+\frac{t}{\nu}
|\xi|^{2(1-\sigma)} }
-1}{
\sqrt{
1-\frac{4}{\nu^{2}}
|\xi|^{2(1-2\sigma)}
}
}
\right|
+
\left|
 \frac{1}{
\sqrt{
1-\frac{4}{\nu^{2}}
|\xi|^{2(1-2\sigma)}
}
}
-1
 \right| \right) \\
& \le C e^{
-c(1+t)
|\xi|^{2(1-\sigma)}
}
|\xi|^{k}
(t |\xi|^{2(2-3\sigma)-2 \sigma} + |\xi|^{2(1-2\sigma)-2 \sigma} ) \chi_{L},
\end{split}
\end{equation*}
which is the desired estimate \eqref{eq:4.2}, and the proof is now complete.
\end{proof}
\begin{lem} \label{lem:4.2}
Let $n \ge 1$, $\sigma \in (0, \frac{1}{2})$ and $k \ge 0$. Then, there exist $C>0$ and $c>0$ such that 
\begin{equation} \label{eq:4.8}
\begin{split}
& 
|\xi|^{k}
\left|
\partial_{t} 
\left(
\J_{1L}(t,\xi)
-
e^{
-\frac{t}{\nu}
|\xi|^{2(1-\sigma)}
}
\right)
\chi_{L} \right| 
\le C e^{
-c(1+t)
|\xi|^{2(1-\sigma)}
}
|\xi|^{k} 
(t |\xi|^{2(3-4\sigma)} + |\xi|^{2(2-3\sigma)}) \chi_{L},
\end{split}
\end{equation}
\begin{equation}
\begin{split}
& 
|\xi|^{k}
\left|
\partial_{t} 
\left(
\J_{2L}(t,\xi)
-
\frac{
e^{
-\frac{t}{\nu}
|\xi|^{2(1-\sigma)
}
}
}{ \nu |\xi|^{2 \sigma}}\chi_{L} \right) \right| \\
& \le C e^{
-c(1+t)
|\xi|^{2(1-\sigma)}
}
|\xi|^{k} 
(t |\xi|^{2(3-5\sigma)} + |\xi|^{4(1-2\sigma)}) \chi_{L},
\label{eq:4.9}
\end{split}
\end{equation}
where $\J_{1L}(t,\xi)$ and $\J_{2L}(t,\xi)$ are defined by \eqref{eq:2.6}.
\end{lem}
\begin{proof}
At first, we show \eqref{eq:4.8}.
%%%%%%%
%%%%%
From the direct calculation,
it is easy to see that
\begin{equation*} 
\begin{split}
\left(
\frac{\lambda_{+} \lambda_{-} }{\lambda_{+} - \lambda_{-}} - \frac{1}{\nu}|\xi|^{2(1-\sigma)}
\right)\chi_{L}
& = \frac{
|\xi|^{2(1-\sigma)}
}{\nu}
\left(
\frac{ 
1
}{ 
\sqrt{
1-\frac{4}{\nu^{2}}
|\xi|^{2(1-2\sigma)}
}
}
-1
\right)\chi_{L} \\
& = 
\frac{ 
-4
|\xi|^{2(2-3\sigma)}
}{ \nu^{3}\sqrt{
1-\frac{4}{\nu^{2}}
|\xi|^{2(1-2\sigma)}
} \left(1+
\sqrt{
1-\frac{4}{\nu^{2}}
|\xi|^{2(1-2\sigma)}
}
\right)
}\chi_{L}
\end{split}
\end{equation*}
and 
\begin{equation} \label{eq:4.10} 
\begin{split}
\left|
\left(
\frac{\lambda_{+} \lambda_{-}}{\lambda_{+} - \lambda_{-}} -\frac{1}{\nu}|\xi|^{2(1-\sigma)}
\right)\chi_{L} \right|
& \le C|\xi|^{2(2-3\sigma)}\chi_{L}
\end{split}
\end{equation}
for small $|\xi|$.
Combining $\lambda_{+} \lambda_{-}= |\xi|^{2}$, \eqref{eq:4.5} and \eqref{eq:4.10}, 
we arrive at  
\begin{align*}
& |\xi|^{k}
\left|
\partial_{t} \left(
\J_{1L}(t,\xi)
-
e^{
-\frac{t}{\nu}
|\xi|^{2(1-\sigma)}
}
\right)
\chi_{L} \right| \\
& \le 
|\xi|^{k} e^{
-\frac{t}{\nu}
|\xi|^{2(1-\sigma)}
} \chi_{L}
\left(
\left|
\frac{-\lambda_{+}\lambda_{-}
(e^{
\lambda_{+}t +\frac{t}{\nu}
|\xi|^{2(1-\sigma)}
} -1)}
{\lambda_{+} -\lambda_{-}} \right| 
+
\left|
\left(
\frac{\lambda_{+}\lambda_{-}}
{\lambda_{+} -\lambda_{-}}
-\frac{1}{\nu}|\xi|^{2(1-\sigma)} \right) \right| 
\right)
\\
& \le C e^{
-c(1+t)
|\xi|^{2(1-\sigma)}
}
|\xi|^{k} 
(t |\xi|^{2(3-4\sigma)} + |\xi|^{2(2-3\sigma)}) \chi_{L},
\end{align*}
which is the desired estimate \eqref{eq:4.8}.
Next, we show \eqref{eq:4.9}.
Again, the direct calculation gives
\begin{equation} \label{eq:4.11}
\begin{split}
& \left|
\frac{\lambda_{+}}{\lambda_{+} - \lambda_{-}}
+ \frac{1}{\nu^{2}} |\xi|^{2(1-2 \sigma)}
\right| \\
& = 
\frac{1}{\nu^{2}} |\xi|^{2(1-2 \sigma)}
\left(
\frac{-2}{
\left(1+ 
\sqrt{
1-\frac{4}{\nu^{2}}
|\xi|^{2(1-2\sigma)}
}
\right)
\sqrt{
1-\frac{4}{\nu^{2}}
|\xi|^{2(1-2\sigma)}
}
}
+1
\right)
\end{split}
\end{equation}
and
\begin{equation} \label{eq:4.12}
\frac{-2}{
\left(1+ 
\sqrt{
1-\frac{4}{\nu^{2}}
|\xi|^{2(1-2\sigma)}
}
\right)
\sqrt{
1-\frac{4}{\nu^{2}}
|\xi|^{2(1-2\sigma)}
}
}
+1
=O(|\xi|^{2(1-2\sigma)})
\end{equation}
as $\vert\xi\vert \to 0$.
Thus by \eqref{eq:4.5}, \eqref{eq:4.11} and \eqref{eq:4.12}, 
we arrive at the estimate 
% 
%%%%%%%%%%%
%
\begin{align*}
& |\xi|^{k}
\left|
\partial_{t} \left(
\J_{2L}(t,\xi)
-
\frac{
e^{
-\frac{t}{\nu}
|\xi|^{2(1-\sigma)}
}
}
{\nu |\xi|^{2 \sigma}}
\right)
\chi_{L} \right| \\
& \le 
|\xi|^{k} e^{
-\frac{t}{\nu}
|\xi|^{2(1-\sigma)}
} \chi_{L}
\left(
\left|
\frac{\lambda_{+}
(e^{
\lambda_{+}t +\frac{t}{\nu}
|\xi|^{2(1-\sigma)}
} -1)}
{\lambda_{+} -\lambda_{-}} \right| 
+
\left|
\left(
\frac{\lambda_{+}}
{\lambda_{+} -\lambda_{-}}
+
\frac{1}{\nu^{2}}
|\xi|^{2(1-2\sigma)}
 \right) \right| 
\right)
\\
& \le C e^{
-c(1+t)
|\xi|^{2(1-\sigma)}
}
|\xi|^{k} 
(t |\xi|^{2(3-5\sigma)} + |\xi|^{4(1-2\sigma)}) \chi_{L},
\end{align*}
%%%%%%%%%%%
%
which is the desired estimate \eqref{eq:4.9}, and the proof is now complete.
\end{proof}

\subsection{The case for $\sigma \in (1/2,1]$}
For the case $\sigma \in (\frac{1}{2}, 1]$, we claim that 
the approximation of $\K_{jL}(t,\xi)$ is given by not only the parabolic kernel 
but also the hybrid of the parabolic kernel and hyperbolic oscillations,
$\cos(t|\xi|)$ and $\frac{\sin(t |\xi|)}{|\xi|}$. 
This point of view is shared by \cite{I}, \cite{ITY}, \cite{P} and \cite{S}  
for $\sigma =1$.
\begin{lem} \label{lem:4.3}
Let $n \ge 1$, $\sigma \in (\frac{1}{2},1]$ and $k \ge 0$.
Then, there exist $C>0$ and $c>0$ such that 
\begin{align} \label{eq:4.13}
&  |\xi|^{k} 
 \left| 
\K_{1L}(t,\xi)
-e^{-\frac{\nu t |\xi|^{2 \sigma}}{2}} \cos (t |\xi|) \chi_{L}
\right| \le C  e^{-c (1+t) |\xi|^{2 \sigma}} t |\xi|^{k+4 \sigma-1}, \\
&  |\xi|^{k}  \left|
\K_{3L}(t,\xi)
- e^{-\frac{\nu t |\xi|^{2 \sigma}}{2}}
\frac{\sin (t |\xi|)}{|\xi|} \chi_{L}
\right| \le C e^{-c (1+t) |\xi|^{2 \sigma}} |\xi|^{k} (t|\xi|^{4 \sigma -2} + |\xi|^{4 \sigma-3}), \label{eq:4.14}
\end{align}
where $\K_{1L}(\xi)$ and $\K_{3L}(\xi)$ are defined by \eqref{eq:2.11}.
\end{lem}
\begin{proof}[Proof of Lemma \ref{lem:4.3}]
At first, we show \eqref{eq:4.13}.
We note that
\begin{align} \label{eq:4.15}
|\phi_{\sigma}(\xi)-1| \le C |\xi|^{4 \sigma-2}, \\
e^{-\frac{\nu t |\xi|^{2 \sigma}}{2}} \le C e^{-c(1+t) |\xi|^{2 \sigma}} \label{eq:4.16}
\end{align}
for $\xi \in \supp \chi_{L}$.
Indeed, 
the direct calculation gives 
\begin{align*}
\phi_{\sigma}(\xi) -1 = -\frac{\nu^{2} |\xi|^{4 \sigma-2}}{4 \left(1+ \phi_{\sigma}(\xi) \right)},
\end{align*}
which shows \eqref{eq:4.15} for small $|\xi|$, 
where $\phi_{\sigma}$ is defined by \eqref{eq:2.7}.
Now we use the mean value theorem to observe that
there exists $\theta \in (0,1)$ such that
\begin{equation*}
\begin{split}
\cos (t  |\xi|\phi_{\sigma}(\xi))-\cos (t |\xi|)
= -t  |\xi|( \phi_{\sigma}(\xi)-1)\ \sin  (t  |\xi| (\theta \phi_{\sigma}(\xi) +(1- \theta ))),
\end{split}
\end{equation*}
and so \eqref{eq:4.15} and \eqref{eq:4.16} give 
\begin{equation} \label{eq:4.17} 
\begin{split}
 |\xi|^{k}  e^{-\frac{\nu t |\xi|^{2}}{2}}
|\cos (t  |\xi|\phi_{\sigma}(\xi))-\cos (t |\xi|) | \chi_{L}
\le C e^{-(1+t)|\xi|^{2 \sigma}}t  |\xi|^{k+4 \sigma -1} \chi_{L},
\end{split}
\end{equation}
which implies the desired estimate \eqref{eq:4.13}.
Next, we prove the estimate \eqref{eq:4.14}.
Here we apply the mean value theorem again to deduce  
\begin{equation*}
\begin{split}
\sin (t  |\xi|\phi_{\sigma}(\xi))-\sin (t |\xi|) 
= t  |\xi|( \phi_{\sigma}(\xi)-1)\ \cos  (t  |\xi| (\theta \phi_{\sigma}(\xi) +(1- \theta )))
\end{split}
\end{equation*}
for some $\theta \in (0,1)$, and so
\begin{equation} \label{eq:4.18} 
\begin{split}
|\sin (t  |\xi|\phi_{\sigma}(\xi))-\sin (t |\xi|) | \chi_{L}\le C t  |\xi|^{4 \sigma -1} \chi_{L},
\end{split}
\end{equation}
by \eqref{eq:4.15}.
Therefore, the combination of \eqref{eq:4.15}, \eqref{eq:4.16} and \eqref{eq:4.18} yields
\begin{equation*}
\begin{split}
& |\xi|^{k}  \left| 
\K_{3L}(t,\xi)
-
e^{-\frac{\nu t |\xi|^{2 \sigma}}{2}}
\frac{\sin (t |\xi|)}{|\xi|} \chi_{L}
\right| \\
& \le C |\xi|^{k} e^{-\frac{\nu t |\xi|^{2\sigma}}{2}}
\left(
\left| 
\frac{\sin (t |\xi| \phi_{\sigma}(\xi))-\sin (t |\xi|)}{|\xi|\phi_{\sigma}(\xi)}
\right|
+ 
\left| 
\sin (t |\xi|)\left(
\frac{1}{|\xi|}
-
\frac{1}{|\xi| \phi_{\sigma}(\xi)}
\right)
\right| 
\right) \chi_{L}
\\
& \le C   |\xi|^{k}  e^{-\frac{\nu t |\xi|^{2 \sigma}}{2}} 
\left(
\left| 
\frac{ t |\xi|^{4 \sigma-1}}{|\xi| \phi_{\sigma}(\xi)}
\right|
+
\left| 
\frac{
 (\phi_{\sigma}(\xi)- 1)
}{|\xi| \phi_{\sigma}(\xi)}
\right| 
\right) \chi_{L}
\\
& \le C  e^{-c(1+t) |\xi|^{2 \sigma}} |\xi|^{k}(t|\xi|^{4 \sigma-2} + |\xi|^{4 \sigma -3}) \chi_{L},
\end{split}
\end{equation*}
which is the desired estimate \eqref{eq:4.14}, and the lemma now follows.
%%%%%%%%%%%%%%%%%%%%%%%%%%%%%%%%%%%%%%
%%%%%%%%%%%%%%%%%%%%%%%%%%%%%%%%%%%%%%%%%%%
\end{proof}
The following lemma states that 
the approximation functions $\partial_{t} \K_{jL}(t,\xi)$ are not simply given by the $t$ derivative of 
the approximation functions $\K_{jL}(t,\xi)$.
\begin{lem} \label{lem:4.4}
Let $n \ge 1$, $\sigma \in (\frac{1}{2},1]$ and $k \ge 0$.
Then, there exist $C>0$ and $c>0$ such that 
\begin{align} \label{eq:4.19}
&  |\xi|^{k} 
 \left| 
\partial_{t}
\K_{1L}(t,\xi)
+e^{-\frac{\nu t |\xi|^{2 \sigma}}{2}} |\xi| \sin (t |\xi|) \chi_{L}
\right| \le Ce^{-c(1+ t) |\xi|^{2 \sigma}}  |\xi|^{k}
(t|\xi|^{4 \sigma}+|\xi|^{2 \sigma} )
\chi_{L}, \\
&  |\xi|^{k}  \left|
\partial_{t} 
\K_{3L}(t,\xi)
- e^{-\frac{\nu t |\xi|^{2 \sigma}}{2}}\cos (t |\xi|) \chi_{L}
\right| \le C e^{-c (1+t) |\xi|^{2 \sigma}} |\xi|^{k} (t|\xi|^{4 \sigma-1} + |\xi|^{2 \sigma-1}), \label{eq:4.20}
\end{align}
where $\K_{1L}(\xi)$ and $\K_{3L}(\xi)$ are defined by \eqref{eq:2.11}.
\end{lem}
\begin{proof}[Proof of Lemma \ref{lem:4.4}]
We first prove \eqref{eq:4.18}.
By direct calculation, 
we have 
\begin{equation*}
\partial_{t} \K_{1L}(t,\xi)
= 
-e^{-\frac{\nu t |\xi|^{2 \sigma}}{2}} 
\left\{
\frac{\nu|\xi|^{2 \sigma}}{2}\cos (t |\xi| \phi_{\sigma}) 
+ |\xi| \phi_{\sigma} \sin (t |\xi| \phi_{\sigma}) 
\right\}
\chi_{L}.
\end{equation*}
Then we use \eqref{eq:4.15} - \eqref{eq:4.18} 
and the fact that $4 \sigma -1> 2 \sigma$ for $\sigma \in (\frac{1}{2},1]$ to obtain 
\begin{equation*}
\begin{split}
& |\xi|^{k} |\partial_{t} \K_{1L}(t,\xi)+e^{-\frac{\nu t |\xi|^{2 \sigma}}{2}} |\xi| \sin (t |\xi|) \chi_{L}| \\
& \le Ce^{-\frac{\nu t |\xi|^{2 \sigma}}{2}}  |\xi|^{k}
(|\xi|^{2 \sigma}
+
|\xi| \phi_{\sigma} |\sin(t |\xi| \phi_{\sigma}) -\sin(t |\xi|)|
+ |\xi| |\sin(t|\xi|)| |\phi_{\sigma}-1|)
\chi_{L} \\
& \le Ce^{-\frac{\nu t |\xi|^{2 \sigma}}{2}}  |\xi|^{k}
(|\xi|^{2 \sigma} + t|\xi|^{4 \sigma} + |\xi|^{4 \sigma-1})
\chi_{L} 
\le Ce^{-c(1+ t) |\xi|^{2 \sigma}}  |\xi|^{k}
(|\xi|^{2 \sigma} + t|\xi|^{4 \sigma})
\chi_{L}, 
\end{split}
\end{equation*}
which is the desired estimate \eqref{eq:4.19}.
Next, we prove \eqref{eq:4.20}.
Again, the direct computation gives
\begin{equation*}
\partial_{t} \K_{3L}(t,\xi)
= 
\K_{1L}(t,\xi)- |\xi|^{2(\sigma-1)}\K_{2L}(t,\xi)
\end{equation*}
and we find 
\begin{equation*}
\partial_{t} \K_{3L}(t,\xi)- e^{-\frac{\nu t |\xi|^{2 \sigma}}{2}}\cos (t |\xi|) \chi_{L}
= 
(\K_{1L}(t,\xi)- e^{-\frac{\nu t |\xi|^{2 \sigma}}{2}}\cos (t |\xi|) \chi_{L})-  |\xi|^{2(\sigma-1)}\K_{2L}(t,\xi).
\end{equation*}
Therefore,  
we see at once the desired estimate \eqref{eq:4.20} from 
\eqref{eq:4.13} and \eqref{eq:3.3} for $\K_{2L}$,
and the proof is now complete.
\end{proof}
\subsection{The case for $\sigma =1/2$.}
In this subsection, we deal with the case $\sigma=\displaystyle{\frac{1}{2}}$ for \eqref{eq:1.1}.
\begin{lem} \label{lem:4.5}
Let $n \ge 1$, $\ell=0,1$ and $k \ge 0$.
Then, there exist $C>0$ and $c>0$ such that
\begin{equation} \label{eq:4.22}
|\xi|^{k} 
(|\partial_{t}^{\ell} \tilde{\J}_{1}(t,\xi)| +| \partial_{t}^{\ell} \tilde{\J}_{2}(t,\xi)|) \le C e^{-ct |\xi|} |\xi|^{\ell+k},
\end{equation}
\begin{equation} \label{eq:4.23}
|\xi|^{k}( |\partial_{t}^{\ell} \tilde{\K}_{1}(t,\xi)| +| \partial_{t}^{\ell} \tilde{\K}_{2}(t,\xi)| )
\le C e^{-ct |\xi|} |\xi|^{\ell+k},
\end{equation}
\begin{equation} \label{eq:4.24}
|\xi|^{k} |\partial_{t}^{\ell} \tilde{\J}_{3}(t,\xi)| \le C t^{1-\ell} e^{-ct |\xi|} |\xi|^{k}, 
\quad |\xi|^{k} |\partial_{t}^{\ell} \tilde{\K}_{3}(t,\xi)| \le C t^{1-\ell} e^{-ct |\xi|} |\xi|^{k},
\end{equation}
where $\tilde{\J}_{j}(t,\xi)$ and $\tilde{\K}_{j}(t,\xi)$ for $j=1,2,3$ are defined by \eqref{eq:2.14} and \eqref{eq:2.15}, respectively.
\end{lem}
\begin{lem} \label{lem:4.6}
Let $n \ge 1$, $\ell=1$ and $k \ge 0$.
Then, there exist $C>0$ and $c>0$ such that
\begin{equation} \label{eq:4.25}
|\xi|^{k} |\partial_{t} \E_{1}(t,\xi)| \le C e^{-ct |\xi|} |\xi|^{k+1},
\end{equation}
\begin{equation} \label{eq:4.26}
|\xi|^{k} |\partial_{t} \E_{2}(t,\xi)| \le C e^{-ct |\xi|} |\xi|^{k+1}  (1+t |\xi|),
\end{equation}
\begin{equation} \label{eq:4.27}
|\xi|^{k} |\partial_{t} \E_{3}(t,\xi)| \le Ce^{-ct |\xi|} |\xi|^{k}  (1+t |\xi|) ,
\end{equation}
where $\E_{j}(t,\xi)$ for $j=1,2,3$ are defined by \eqref{eq:2.18}.
\end{lem}
\begin{proof}[Proof of Lemmas \ref{lem:4.5} and \ref{lem:4.6}]
\eqref{eq:4.22} - \eqref{eq:4.27} are shown by the similar way. So, we only show \eqref{eq:4.24} for $\tilde{\K}_{3}(t,\xi)$ with $\ell=0$.
Recalling the fact that 
\begin{equation*}
\left| \frac{\sin y}{y} \right| \le 1
\end{equation*}
for $y \in \R$, 
we see 
\begin{equation*}
|\xi|^{k} \left| 
\K_{3}(t,\xi)
\right|
= 
t |\xi|^{k} \left| 
\frac{
e^{-\frac{\nu |\xi| t}{2}}
}
{ \frac{t |\xi|
\sqrt{4-\nu^{2}} }
{2}
}
\sin \left( 
\frac{t |\xi| \sqrt{4-\nu^{2}} }{2}
\right)
\right| \le C e^{-ct|\xi|}t |\xi|^{k},
\end{equation*}
which is the desired estimate \eqref{eq:4.24} for $\tilde{\K}_{3}(t,\xi)$ with $\ell=0$.
We complete the proof of Lemmas \ref{lem:4.5} and \ref{lem:4.6}.
\end{proof}
\section{Decay properties of the localized evolution operators}
%%%%
%%%
In this section, we prove several decay properties of the localized evolution operators $J_{jk}(t)g$ 
for $j=1,2,3,4$, $k=L,M,H$ and $K_{jk}(t)g$ 
for $j=1,2,3$, $k=L,M,H$,
by using point-wise estimates of the Fourier multipliers.

\subsection{Preliminaries}
In this subsection, we present useful estimates to obtain some decay estimates of the evolution operators.
The estimates presented here are frequently used throughout this section and next section. 

We begin with the simple application of the H\"older inequality (cf. \cite{IT}).
%%%%
%%%
\begin{lem} \label{Lem:5.1}
Let $n \ge 1$, $1 \le r \le 2$ and $\displaystyle{\frac{1}{r}} + \displaystyle{\frac{1}{r'}} = 1$.
Then it holds that 
\begin{equation} \label{eq:5.1}
\| f g \|_{2} \le \| f \|_{\frac{2r}{2-r}} \| g \|_{r'}.
\end{equation}
\end{lem}
The following lemma is useful to obtain a sharp decay property of the Fourier multipliers.
\begin{lem} \label{lem:5.2}
Let $n \ge 1$,  $C_{0}>0$, $1 \le r \le 2$,  $s>0$, $\alpha>0$ and $\beta \ge 0$. 
Then it holds that
\begin{align} \label{eq:5.2}
& \| e^{-C_{0} s |\xi|^{\alpha}} |\xi|^{\beta} \chi_{L} \|_{L^{\frac{2r}{2-r}}(\R^{n})}
\le Cs^{-\frac{n}{\alpha}(\frac{1}{r} -\frac{1}{2})-\frac{\beta}{\alpha}}, \\
& \| e^{-C_{0} s|\xi|^{\alpha}} |\xi|^{\beta} (\chi_{M} +\chi_{H}) \|_{L^{\frac{2r}{2-r}}(\R^{n})}
\le C e^{-c s}s^{-\frac{n}{\alpha}(\frac{1}{r} -\frac{1}{2})-\frac{\beta}{\alpha}}, \label{eq:5.3}
\end{align}
where the constants $C>0$ and $c>0$ are independent of $s$.
\end{lem}
\begin{proof}
Let us first prove \eqref{eq:5.2}.
Changing the integral variable $\eta= (C_{0} s \displaystyle{\frac{2r}{2-r}})^{\frac{1}{\alpha}} \xi$,
we see
\begin{equation*}
\begin{split}
\| e^{-C_{0} s |\xi|^{\alpha}} |\xi|^{\beta} \chi_{L} \|_{\frac{2r}{2-r}}^{\frac{2r}{2-r}}
& = \int_{\R^{n}}
e^{-C_{0} s |\xi|^{\alpha} \frac{2r}{2-r}} |\xi|^{\frac{\beta r}{2-r}} \chi_{L}^{\frac{2r}{2-r}} d \xi \\
& = C 
s^{-\frac{n}{\alpha}-\frac{2 \beta r}{\alpha(2-r)}}
\int_{\R^{n}}
e^{-|\eta|^{\alpha}} |\eta|^{\frac{\beta r}{2-r}} \chi_{L}^{\frac{2r}{2-r}} d \eta,
\end{split}
\end{equation*}
and so that 
\begin{equation*}
\begin{split}
\| e^{-C_{0} s |\xi|^{\alpha}} |\xi|^{\beta} \chi_{L} \|_{\frac{2r}{2-r}}^{\frac{2r}{2-r}}
& \le C 
s^{-\frac{n}{\alpha}-\frac{2 \beta r}{\alpha(2-r)}},
\end{split}
\end{equation*}
which is the desired estimate \eqref{eq:5.2}.
By a similar computation, we easily have \eqref{eq:5.3}. The proof is now complete.
\end{proof}
\subsection{The case for $\sigma \in (0,1/2)$.}
The localized evolution operators $J_{jL}(t)g$ ($j=1,2,3,4$) are estimated as follows. 
%%%%%%%%%%
%%%%%%%%%%%%%
\begin{lem} \label{lem:5.3}
Let $n \ge 1$, $k \ge \tilde{k} \ge 0$, $\ell=0,1$, $1 \le r \le 2$, 
$\nu>0$ and $\sigma \in (0, \displaystyle{\frac{1}{2}})$. 
Then it holds that
\begin{align} \label{eq:5.4}
& \left\| 
\partial_{t}^{\ell} 
\nabla_{x}^{k}
J_{1L}(t)g
\right\|_{2}
\le C(1+t)^{-\frac{n}{2(1-\sigma)}(\frac{1}{r}-\frac{1}{2})- \ell- \frac{k-\tilde{k}}{2(1-\sigma)}}
\| \nabla^{\tilde{k}}_{x} g \|_{r},\\
& \left\| 
\partial_{t}^{\ell} 
\nabla_{x}^{k}
J_{3L}(t)g
\right\|_{2}
\le C(1+t)^{-\frac{n}{2\sigma}(\frac{1}{r}-\frac{1}{2})
-\frac{1-2\sigma}{\sigma} - \ell- \frac{k-\tilde{k}}{2\sigma}}
\| \nabla^{\tilde{k}}_{x} g \|_{r}, \label{eq:5.5} 
\end{align}
where $J_{jL}(t)g$ for $j=1,3$ are defined by \eqref{eq:2.5}.
\end{lem}
\begin{lem} \label{lem:5.4}
Let $n \ge 2$, $k \ge \tilde{k} \ge 0$, $\ell=0,1$, $1 \le r \le 2$, 
$\nu>0$ and $\sigma \in (0, \displaystyle{\frac{1}{2}})$. 
Then it holds that
\begin{align}
& 
\left\|
\partial_{t}^{\ell} 
\nabla_{x}^{k}
J_{2L}(t)g 
\right\|_{2} 
\le C(1+t)^{-\frac{n}{2(1-\sigma)}(\frac{1}{r}-\frac{1}{2})
+\frac{\sigma}{1-\sigma} - \ell- \frac{k-\tilde{k}}{2(1-\sigma)}}
\| \nabla^{\tilde{k}}_{x} g \|_{r}, \label{eq:5.6} \\
& 
\left\|
\partial_{t}^{\ell} 
\nabla_{x}^{k}
J_{4L}(t)g 
\right\|_{2} 
\le C(1+t)^{-\frac{n}{2(1-\sigma)}(\frac{1}{r}-\frac{1}{2})+1- \ell- \frac{k-\tilde{k}}{2\sigma}}
\| \nabla^{\tilde{k}}_{x} g \|_{r}, \label{eq:5.7}
\end{align}
where $J_{jL}(t)g$ for $j=2,4$ are defined by \eqref{eq:2.5}.
\end{lem}
\begin{proof}[Proof of Lemmas \ref{lem:5.3} and \ref{lem:5.4}]
We can show \eqref{eq:5.4} - \eqref{eq:5.7} by the similar way.
Here we only prove \eqref{eq:5.4}.
We apply the Planchrel formula and \eqref{eq:3.1}, \eqref{eq:5.1} and \eqref{eq:5.2} 
with $C_{0}=c$, $s=1+t$, $\alpha=2(1-\sigma)$ and $\beta= 2(1-\sigma) \ell+ k-\tilde{k}$ to have
\begin{equation*}
\begin{split}
\left\| 
\partial_{t}^{\ell} 
\nabla_{x}^{k}
J_{1L}(t)g
\right\|_{2} 
& \le C \| e^{
-c(1+t)|\xi|^{2(1-\sigma)}
}
|\xi|^{2(1-\sigma) \ell +k-\tilde{k}}  \chi_{L} |\xi|^{\tilde{k}} \hat{g} \|_{2} \\
& \le C 
\| e^{
-c(1+t)|\xi|^{2(1-\sigma)}
}
|\xi|^{2(1-\sigma) \ell +k-\tilde{k}}  \chi_{L} \|_{\frac{2r}{2-r}} 
\| |\xi|^{\tilde{k}} \hat{g} \|_{r'} \\
& \le C(1+t)^{-\frac{n}{2(1-\sigma)}(\frac{1}{r}-\frac{1}{2})- \ell- \frac{k-\tilde{k}}{2(1-\sigma)}}
\| \nabla^{\tilde{k}}_{x} g \|_{r},
\end{split}
\end{equation*}
which is the desired estimate \eqref{eq:5.4}, and the lemma follows. 
\end{proof}
The following lemma suggests that the localized operators in the middle and high frequency parts decay exponentially, 
and we see that their effect is negligible in the large time behavior case. 
\begin{lem}  \label{lem:5.5}
Let $n \ge 1$,  $\ell=0,1$, $k+\ell \ge \tilde{k} \ge 0$, $1 \le r \le 2$, 
$\nu>0$ and $\sigma \in (0, \displaystyle{\frac{1}{2}})$. 
Then it holds that 
\begin{align} \label{eq:5.8}
& 
\sum_{j=1,3}
(
\|
\partial_{t}^{\ell} \nabla_{x}^{k}J_{jM}(t)g
\|_{2}
+
\|
\partial_{t}^{\ell} \nabla_{x}^{k} J_{jH}(t)g
\|_{2}
)
\le C e^{-ct} t^{-\frac{n}{2 \sigma}(\frac{1}{r}-\frac{1}{2})-\frac{k+\ell-\tilde{k}}{2 \sigma} }
\| \nabla^{\tilde{k}}_{x} g\|_{r}, \\
& 
\sum_{j=2,4}
(
\|
\partial_{t}^{\ell} \nabla_{x}^{k}J_{jM}(t)g
\|_{2}
+
\|
\partial_{t}^{\ell} \nabla_{x}^{k} J_{jH}(t)g
\|_{2}
)
\le C e^{-ct} t^{-\frac{n}{2 \sigma}(\frac{1}{r}-\frac{1}{2})-\frac{k+\ell-\tilde{k}}{2 \sigma} }
\| \nabla^{(\tilde{k}-1)_{+}}_{x} g\|_{r},  \label{eq:5.9}
\end{align}
where $J_{jL}(t)g$ for $j=1,2,3,4$ are defined by \eqref{eq:2.5}, 
and $(\tilde{k}-1)_{+} = \max\{\tilde{k}-1, 0 \}$.
\end{lem}
\begin{proof}
We now apply the argument of the proof of Lemma \ref{lem:5.3}, 
with \eqref{eq:5.2} replaced by \eqref{eq:5.3}, to obtain \eqref{eq:5.8} and \eqref{eq:5.9}.
We now complete the proof of Lemma \ref{lem:5.5}.  
\end{proof}

\subsection{The case for $\sigma \in (1/2,1]$}
For the case $\sigma \in (\frac{1}{2}, 1]$, 
we have the following decay property of the localized operators defined in the low frequency region. 
\begin{lem} \label{lem:5.6}
Let $n \ge 1$, $\ell=0,1$,  $k + \ell \ge \tilde{k} \ge 0$, $1 \le r \le 2$,
$\nu>0$ and $\sigma \in (\displaystyle{\frac{1}{2}},1]$. 
Then it holds that
\begin{align} \label{eq:5.10}
& \left\| 
\partial_{t}^{\ell} 
\nabla_{x}^{k}
K_{1L}(t)g
\right\|_{2}
\le C(1+t)^{-\frac{n}{2 \sigma }(\frac{1}{r}-\frac{1}{2})- \frac{\ell+ k-\tilde{k}}{2\sigma}}
\| \nabla^{\tilde{k}}_{x} g \|_{r},\\
& 
\left\|
\partial_{t}^{\ell} 
\nabla_{x}^{k}
K_{2L}(t)g 
\right\|_{2} 
\le C(1+t)^{-\frac{n}{2 \sigma }(\frac{1}{r}-\frac{1}{2})- \frac{\ell+ k-\tilde{k}+1}{2\sigma}}
\| \nabla^{\tilde{k}}_{x} g \|_{r}, \label{eq:5.11}  
\end{align}
where $K_{jL}(t)g$ for $j=1,2$ are defined by \eqref{eq:2.10}.
\end{lem}
\begin{lem} \label{lem:5.7}
Let $n \ge 3$, $\ell=0,1$, $k+ \ell \ge \tilde{k} \ge 0$, $1 \le r \le 2$, 
$\nu>0$ and $\sigma \in (\displaystyle{\frac{1}{2}},1]$. 
Then it holds that
\begin{align} 
\left\| 
\partial_{t}^{\ell} 
\nabla_{x}^{k}
K_{3L}(t)g
\right\|_{2}
\le C(1+t)^{-\frac{n}{2 \sigma }(\frac{1}{r}-\frac{1}{2})- \frac{\ell+ k-\tilde{k}-1}{2\sigma}}
\| \nabla^{\tilde{k}}_{x} g \|_{r}, \label{eq:5.12} 
\end{align}
where $K_{3L}(t)g$ is defined by \eqref{eq:2.10}.
\end{lem}
\begin{rem} 
In Lemma \ref{lem:5.7}, 
if we assume $\ell +k > 3-n$, then \eqref{eq:5.12} is also valid for all $n \ge 1$.
\end{rem}
\begin{proof}[Proof of Lemmas \ref{lem:5.5} and \ref{lem:5.6}]
We note that \eqref{eq:5.10} - \eqref{eq:5.12} are shown by the similar way, so we only prove \eqref{eq:5.10}.
We apply the Plancherel formula and \eqref{eq:3.3}, \eqref{eq:5.1} and \eqref{eq:5.2} 
with $C_{0}=c$, $s=1+t$, $\alpha=2\sigma$ and $\beta= \ell+ k-\tilde{k}$ to obtain
\begin{equation*}
\begin{split}
\left\| 
\partial_{t}^{\ell} 
\nabla_{x}^{k}
K_{1L}(t)g
\right\|_{2} 
& \le C \| e^{
-c(1+t)|\xi|^{2\sigma}
}
|\xi|^{\ell +k-\tilde{k}}  \chi_{L} |\xi|^{\tilde{k}} \hat{g} \|_{2} \\
& \le C 
\| e^{
-c(1+t)|\xi|^{2(1-\sigma)}
}
|\xi|^{\ell +k-\tilde{k}}  \chi_{L} \|_{\frac{2r}{2-r}} 
\| |\xi|^{\tilde{k}} \hat{g} \|_{r'} \\
& \le C(1+t)^{-\frac{n}{2 \sigma }(\frac{1}{r}-\frac{1}{2})- \frac{\ell+ k-\tilde{k}}{2\sigma}}
\| \nabla^{\tilde{k}}_{x} g \|_{r},
\end{split}
\end{equation*}
which is the desired conclusion \eqref{eq:5.10}.
This proves the lemma. 
\end{proof}
The following lemma asserts that the operators localized in the middle and high frequency regions 
do not affect the asymptotic profile of the solution to \eqref{eq:1.1}
because of the exponential decay property
as shown in Lemma \ref{lem:5.5}.
\begin{lem}  \label{lem:5.9}
Let $n \ge 1$, $k \ge \tilde{k} \ge 0$, $\ell=0,1$, $1 \le r \le 2$, 
$\nu>0$ and $\sigma \in (\displaystyle{\frac{1}{2}},1]$. 
Then it holds that 
\begin{equation}
\begin{split} 
\label{eq:5.13}
&
\sum_{j=1,2}
(
\|
\partial_{t}^{\ell} \nabla_{x}^{k}K_{jM}(t)g
\|_{2}
+
\|
\partial_{t}^{\ell} \nabla_{x}^{k} K_{jH}(t)g
\|_{2}
) \\
& \le C e^{-ct} 
t^{-\frac{n}{2(1- \sigma)}(\frac{1}{r}-\frac{1}{2})-\frac{k-\tilde{k}}{2(1- \sigma)} }
\| \nabla^{\tilde{k}+ 2(1-\sigma) \ell}_{x} g\|_{r}, 
\end{split}
\end{equation}
\begin{align} 
\|
\partial_{t}^{\ell} \nabla_{x}^{k}K_{3M}(t)g
\|_{2}
+
\|
\partial_{t}^{\ell} \nabla_{x}^{k} K_{3H}(t)g
\|_{2} 
\le C e^{-ct} 
t^{-\frac{n}{2(1- \sigma)}(\frac{1}{r}-\frac{1}{2})-\frac{k-\tilde{k}}{2(1- \sigma)} }
\| \nabla^{(\tilde{k}-2\sigma (1- \ell))_{+}}_{x} g\|_{r},  \label{eq:5.14}
\end{align}
where $K_{jM}(t)g$ and $K_{jH}(t)g$ for $j=1,2,3$ are defined by \eqref{eq:2.10}, and 
$(\tilde{k}-2\sigma(1-\ell))_{+} = \max\{\tilde{k}-2 \sigma (1- \ell), 0\}$.
\end{lem}
\begin{proof}
Lemma \ref{lem:5.9} is also just an application of \eqref{eq:5.1} and \eqref{eq:5.2}, so we omit its proof. 
\end{proof}
%%%%%%
%%%%%%%%%%
%
\section{Estimates for the evolution operators} \label{sec:6}
In this section, 
by using the point-wise estimates developed in previous section, 
we prove the approximation formulas for operators localized in the low frequency region.
By combining estimates for the middle and high frequency parts and such estimates for the low frequency region, 
we show the asymptotic behavior of the evolution operators of \eqref{eq:1.1}. 

\subsection{Approximation of the operators localized near low frequency region for $\sigma \in (0, 1/2)$}
In this subsection our aim is to show the following proposition, which states that 
the evolution operators $J_{1}(t)$ and $J_{2}(t)$ are approximated by the operators 
$\mathcal{F}^{-1}[e^{
-\frac{t}{\nu}
|\xi|^{2(1-\sigma)}
}]$ and $\mathcal{F}^{-1}[\frac{
e^{
-\frac{t}{\nu}
|\xi|^{2(1-\sigma)
}
}
}{ \nu |\xi|^{2 \sigma}}]$, respectively. 
\begin{lem}
Let $n \ge 1$, $k \ge \tilde{k} \ge 0$, $\ell=0,1$, $1 \le r \le 2$, 
$\nu>0$ and $\sigma \in (0, \displaystyle{\frac{1}{2}})$. 
Then it holds that
\begin{equation} \label{eq:6.1}
\begin{split}
& \left\| 
\partial_{t}^{\ell} 
\nabla_{x}^{k}
\left(
J_{1L}(t)g
-
\mathcal{F}^{-1}
\left[
e^{
-\frac{t}{\nu}
|\xi|^{2(1-\sigma)}
}
\chi_{L}
\right] \ast g
\right)
\right\|_{2} \\
& \le C(1+t)^{-\frac{n}{2(1-\sigma)}(\frac{1}{r}-\frac{1}{2})
-\frac{1-2\sigma}{1-\sigma} - \ell- \frac{k-\tilde{k}}{2(1-\sigma)}}
\| \nabla^{\tilde{k}}_{x} g \|_{r},
\end{split}
\end{equation}
\begin{equation} \label{eq:6.2}
\begin{split}
& 
\left\|
\partial_{t}^{\ell} 
\nabla_{x}^{k}
\left(
J_{2L}(t)g 
-
\mathcal{F}^{-1} 
\left[
\frac{
e^{
-\frac{t}{\nu}
|\xi|^{2(1-\sigma)}
}
}{ \nu |\xi|^{2 \sigma}}\chi_{L} 
\right] \ast g 
\right) 
\right\|_{2} \\
& \le C(1+t)^{-\frac{n}{2(1-\sigma)}(\frac{1}{r}-\frac{1}{2})
-\frac{1-3\sigma}{1-\sigma} - \ell- \frac{k-\tilde{k}}{2(1-\sigma)}}
\| \nabla^{\tilde{k}}_{x} g \|_{r},
\end{split}
\end{equation}
where $J_{1L}(t)g$ and $J_{2L}(t)g$ are defined by \eqref{eq:2.4}.
\end{lem}
\begin{proof}
We first show \eqref{eq:6.1}.
By \eqref{eq:4.1} and \eqref{eq:4.8}, 
we see 
\begin{equation} \label{eq:6.3}
\begin{split}
& 
\left|
\partial_{t}^{\ell} 
\left(
\J_{1L}(t,\xi)
-
e^{
-\frac{t}{\nu}
|\xi|^{2(1-\sigma)}
}
\right)
\chi_{L} \hat{g} \right| \\
& \le C e^{
-c(1+t)
|\xi|^{2(1-\sigma)}
}
|\xi|^{k} 
(t |\xi|^{2(2-3\sigma)+ 2 \ell(1-\sigma)} + |\xi|^{2(1-2\sigma)+2 \ell(1-\sigma)}) \chi_{L} |\hat{g}|,
\end{split}
\end{equation}
for $\ell =0,1$.
Then 
taking the $\| \nabla_{x}^{k}  \cdot \|_{2}$ norm for the both sides of \eqref{eq:6.3}
and applying \eqref{eq:5.1} and \eqref{eq:5.2} with 
$C_{0}=c$, $s=1+t$, $\alpha= 2(1-\sigma)$, $\beta= 2(2-3\sigma) + 2 \ell (1-\sigma) +k-\tilde{k}$ for the 
first factor,
 and  $\beta= 2(1-2\sigma) + 2 \ell (1-\sigma) +k-\tilde{k}$ for the second factor,
we have 
\begin{align*}
& \left\| 
\partial_{t}^{\ell} 
\nabla_{x}^{k}
\left(
J_{1L}(t)g
-
\mathcal{F}^{-1}
\left[
e^{
-\frac{t}{\nu}
|\xi|^{2(1-\sigma)}
}
\chi_{L}
\right] \ast g
\right)
\right\|_{2} \\
& = C
\left\| |\xi|^{k}
\partial_{t}^{\ell} 
\left(
\J_{1L}(t,\xi)
-
e^{
-\frac{t}{\nu}
|\xi|^{2(1-\sigma)}
}
\chi_{L}
\right)
 \hat{g} \right\|_{2} \\
& \le 
Ct
\left\|
 e^{
-c(1+t)
|\xi|^{2(1-\sigma)}
}
|\xi|^{2(2-3\sigma)+ 2 \ell(1-\sigma)+k-\tilde{k}} 
\chi_{L} |\xi|^{\tilde{k} } \hat{g} \right\|_{2} \\
& +C
\left\|
 e^{
-c(1+t)
|\xi|^{2(1-\sigma)}
}
 |\xi|^{2(1-2\sigma)+2 \ell(1-\sigma)+k-\tilde{k}} 
\chi_{L} |\xi|^{\tilde{k} }\hat{g} \right\|_{2} \\
& \le 
Ct
\left\|
 e^{
-c(1+t)
|\xi|^{2(1-\sigma)}
}
|\xi|^{2(2-3\sigma)+ 2 \ell(1-\sigma)+k-\tilde{k}} 
\chi_{L} \|_{\frac{2r}{2-r} }
\| |\xi|^{\tilde{k} } \hat{g} \right\|_{r'} \\
& +C
\left\|
 e^{
-c(1+t)
|\xi|^{2(1-\sigma)}
}
 |\xi|^{2(1-2\sigma)+2 \ell(1-\sigma)+k-\tilde{k}} 
\chi_{L} \|_{\frac{2r}{2-r}}
\| |\xi|^{\tilde{k} }\hat{g} \right\|_{r'} \\
& \le C(1+t)^{-\frac{n}{2(1-\sigma)}(\frac{1}{r}-\frac{1}{2})
-\frac{1-2\sigma}{1-\sigma} - \ell- \frac{k-\tilde{k}}{1-\sigma}}
\| \nabla^{\tilde{k}}_{x} g \|_{r},
\end{align*}
which is the desired estimate \eqref{eq:6.1}.
Next, we prove \eqref{eq:6.2}.
\eqref{eq:4.2} and \eqref{eq:4.9} mean that
\begin{equation}
\begin{split}
& 
\left|
\partial_{t}^{\ell} 
\left(
\J_{2L}(t,\xi)
-
\frac{
e^{
-\frac{t}{\nu}
|\xi|^{2(1-\sigma)
}
}
}{ \nu |\xi|^{2 \sigma}}\chi_{L} \right) \hat{g} \right| \\
& \le C e^{
-c(1+t)
|\xi|^{2(1-\sigma)}
}
(t |\xi|^{2(2-3\sigma)+2 \ell(1-\sigma)-2 \sigma} + |\xi|^{2(1-2\sigma)+2 \ell(1-\sigma) -2 \sigma}) 
\chi_{L} |\hat{g}|,
\label{eq:6.4}
\end{split}
\end{equation}
for $\ell =0,1$.
Therefore we again apply $\| \nabla^{k}_{x} \cdot \|_{2}$ norm for the both sides of \eqref{eq:6.4} and use 
\eqref{eq:5.1} and \eqref{eq:5.2} with $C_{0}=c$, $s=1+t$, $\alpha= 2(1-\sigma)$, $\beta= 2(2-3\sigma) + 2 \ell (1-\sigma) -2 \sigma+k-\tilde{k}$ for the 
first factor, and  $\beta= 2(1-2\sigma) + 2 \ell (1-\sigma) -2 \sigma +k-\tilde{k}$ for the second factor,
to see 
\begin{equation*} 
\begin{split}
& 
\left\|
\partial_{t}^{\ell} 
\nabla_{x}^{k}
\left(
J_{2L}(t)g 
-
\mathcal{F}^{-1} 
\left[
\frac{
e^{
-\frac{t}{\nu}
|\xi|^{2(1-\sigma)}
}
}{ \nu |\xi|^{2 \sigma}}\chi_{L} 
\right] \ast g 
\right) 
\right\|_{2} \\
& \le C \left\| 
|\xi|^{k}
\partial_{t}^{\ell} 
\left(
\J_{2L}(t,\xi)
-
\frac{
e^{
-\frac{t}{\nu}
|\xi|^{2(1-\sigma)
}
}
}{ \nu |\xi|^{2 \sigma}}\chi_{L} \right)  \hat{g}
\right\|_{2} \\
& \le C
t \| 
e^{
-c(1+t)
|\xi|^{2(1-\sigma)}
}
|\xi|^{2(2-3\sigma)+2 \ell(1-\sigma)-2 \sigma+k-\tilde{k}} 
\chi_{L}
|\xi|^{\tilde{k}} \hat{g}
\|_{2} \\
& + \|
e^{
-c(1+t)
|\xi|^{2(1-\sigma)}
}
|\xi|^{2(1-2\sigma)+2 \ell(1-\sigma) -2 \sigma+k-\tilde{k}}
\chi_{L}
|\xi|^{\tilde{k}} \hat{g}
\|_{2} \\
& \le 
Ct
\left\|
 e^{
-c(1+t)
|\xi|^{2(1-\sigma)}
}
|\xi|^{2(2-3\sigma)+ 2 \ell(1-\sigma)+k-\tilde{k}} 
\chi_{L} \|_{\frac{2r}{2-r} }
\| |\xi|^{\tilde{k} } \hat{g} \right\|_{r'} \\
& +C
\left\|
 e^{
-c(1+t)
|\xi|^{2(1-\sigma)}
}
 |\xi|^{2(1-2\sigma)+2 \ell(1-\sigma)-2 \sigma+k-\tilde{k}} 
\chi_{L} \|_{\frac{2r}{2-r}}
\| |\xi|^{\tilde{k} }\hat{g} \right\|_{r'} \\
& \le C(1+t)^{-\frac{n}{2(1-\sigma)}(\frac{1}{r}-\frac{1}{2})
-\frac{1-3\sigma}{1-\sigma} - \ell- \frac{k-\tilde{k}}{2(1-\sigma)}}
\| \nabla^{\tilde{k}}_{x} g \|_{r},
\end{split}
\end{equation*}
which is the desired estimate \eqref{eq:6.2}, and the lemma follows.  
\end{proof}
%%%%
The following lemma is the estimates for the error factor in $\sigma \in (0, \displaystyle{\frac{1}{2}})$, 
which are direct consequences of \eqref{eq:5.1} and \eqref{eq:5.3}.
%%%%
\begin{lem}
Let $n \ge 1$, $k \ge \tilde{k} \ge 0$, $\ell=0,1$, $1 \le r \le 2$, 
$\nu>0$ and $\sigma \in (0, \displaystyle{\frac{1}{2}})$. 
Then it holds that
\begin{align} \label{eq:6.5}
& \left\| 
\partial_{t}^{\ell} 
\nabla_{x}^{k}
\mathcal{F}^{-1}
\left[
e^{
-\frac{t}{\nu}
|\xi|^{2(1-\sigma)}
}
(\chi_{M}+\chi_{H})
\right] \ast g
\right\|_{2}  \le Ce^{-ct}t^{-\frac{n}{2(1-\sigma)}(\frac{1}{r}-\frac{1}{2})- \frac{k-\tilde{k}}{2(1-\sigma)}}
\| \nabla^{\tilde{k}+2 \ell(1-\sigma)}_{x} g \|_{r}, \\
& \left\|
\partial_{t}^{\ell} 
\nabla_{x}^{k}
\mathcal{F}^{-1} 
\left[
\frac{
e^{
-\frac{t}{\nu}
|\xi|^{2(1-\sigma)}
}
}{ \nu |\xi|^{2 \sigma}}(\chi_{M}+\chi_{H})
\right] \ast g 
\right\|_{2} \le Ce^{-ct}t^{-\frac{n}{2(1-\sigma)}(\frac{1}{r}-\frac{1}{2})
+\frac{\sigma}{1-\sigma} - \ell- \frac{k-\tilde{k}}{2(1-\sigma)}}
\| \nabla^{\tilde{k}}_{x} g \|_{r}. \label{eq:6.6}
\end{align}
\end{lem}
\begin{proof}
Since the proof of \eqref{eq:6.5} and \eqref{eq:6.6} is similar, we only show \eqref{eq:6.5}. 
Indeed, \eqref{eq:5.1}, and \eqref{eq:5.3} with $C_{0}= \frac{1}{\nu}$, $s=t$, $\alpha=2(1-\sigma)$,
$\beta=k-\tilde{k}$ give
\begin{align*}
& \left\| 
\partial_{t}^{\ell} 
\nabla_{x}^{k}
\mathcal{F}^{-1}
\left[
e^{
-\frac{t}{\nu}
|\xi|^{2(1-\sigma)}
}
(\chi_{M}+\chi_{H})
\right] \ast g
\right\|_{2}  \\
& \le C 
\left\| 
|\xi|^{k-\tilde{k}}
e^{
-\frac{t}{\nu}
|\xi|^{2(1-\sigma)}
}
(\chi_{M}+\chi_{H})
|\xi|^{\tilde{k}+2 \ell(1-\sigma)} \hat{g}
\right\|_{2} \\
& \le C \left\| 
|\xi|^{k-\tilde{k}}
e^{
-\frac{t}{\nu}
|\xi|^{2(1-\sigma)}
}
(\chi_{M}+\chi_{H})
\right\|_{\frac{2r}{2-r}}
\| |\xi|^{\tilde{k}+2 \ell (1-\sigma)} \hat{g} \|_{r'}
\\
& \le Ce^{-ct}t^{-\frac{n}{2(1-\sigma)}(\frac{1}{r}-\frac{1}{2})
- \frac{k-\tilde{k}}{2(1-\sigma)}}
\| \nabla^{\tilde{k}+2 \ell (1-\sigma)}_{x} g \|_{r},
\end{align*}
which is the desired estimate \eqref{eq:6.5}, and the proof is now complete.
\end{proof}
%%%%%%%%%%%

%%%%%%%%%%%%

\subsection{Approximation of the operator localized near low frequency part for $\sigma \in (1/2,1]$}
The operators $K_{1L}(t)g$ and $K_{3L}(t)g$ for $\sigma \in (1/2,1]$ localized near low frequency area  
are approximated by $\mathcal{F}^{-1} \left[
e^{-\frac{\nu t |\xi|^{2 \sigma}}{2}} \cos (t |\xi|) \chi_{L} \right] \ast g$
and 
$ \mathcal{F}^{-1} \left[e^{-\frac{\nu t |\xi|^{2 \sigma}}{2}}
\frac{\sin (t |\xi|)}{|\xi|} \chi_{L} \right] \ast g$
for large $t$, respectively.
\begin{lem} \label{lem:6.3}
Let $n \ge 1$, $k \ge \tilde{k} \ge 0$, $1 \le r \le 2$, 
$\nu>0$ and $\sigma \in (\displaystyle{\frac{1}{2}},1]$. 
Then, there exist $C>0$ and $c>0$ such that 
\begin{equation} 
\begin{split}
\label{eq:6.7}
&  \left\| 
\nabla^{k}_{x} \left(
 K_{1L}(t)g
-
\mathcal{F}^{-1} \left[
e^{-\frac{\nu t |\xi|^{2 \sigma}}{2}} \cos (t |\xi|) \chi_{L} \right] \ast g
\right)
\right\|_{2} \\
& \le C (1+t)^{-\frac{n}{2 \sigma}(\frac{1}{r}-\frac{1}{2})-\frac{k-\tilde{k}}{2 \sigma}-1+\frac{1}{2 \sigma} }  
\| \nabla^{\tilde{k}}_{x} g \|_{r},
\end{split}
\end{equation}
\begin{equation} 
\begin{split}
& \left\| 
\nabla^{k}_{x} \left(
K_{3L}(t,\xi)
- \mathcal{F}^{-1} \left[e^{-\frac{\nu t |\xi|^{2 \sigma}}{2}}
\frac{\sin (t |\xi|)}{|\xi|} \chi_{L} \right] \ast g
\right)
\right\|_{2} \\
& \le C (1+t)^{-\frac{n}{2 \sigma}(\frac{1}{r}-\frac{1}{2}) -1-\frac{k-\tilde{k}}{2 \sigma} +\frac{1}{\sigma}}
\| \nabla^{\tilde{k}}_{x} g \|_{r},
 \label{eq:6.8}
\end{split}
\end{equation}
where $K_{1L}(t)g$ and $K_{3L}(t)g$ are defined by \eqref{eq:2.10}.
\end{lem}
%%%%
\begin{proof}
At first, we prove \eqref{eq:6.7}.
We apply the estimates \eqref{eq:5.1}, \eqref{eq:4.13} 
and \eqref{eq:5.2} with $C_{0}=c$, $\alpha= 2\sigma$, $\beta=k-\tilde{k}+4 \sigma-1$
to see
\begin{equation*} 
\begin{split}
&  \left\| 
\nabla^{k}_{x} \left(
 K_{1L}(t)g
-
\mathcal{F}^{-1} \left[
e^{-\frac{\nu t |\xi|^{2 \sigma}}{2}} \cos (t |\xi|) \chi_{L} \right] \ast g
\right)
\right\|_{2} \\
& \le C  \| e^{-c (1+t) |\xi|^{2 \sigma}} t |\xi|^{k+4 \sigma-1} \chi_{L}\hat{g} \|_{2} \\
& \le C  \| e^{-c (1+t) |\xi|^{2 \sigma}} t |\xi|^{k-\tilde{k}+4 \sigma-1} \chi_{L} 
\|_{\frac{2r}{2-r}} \| |\xi|^{\tilde{k}} \hat{g} \|_{r'} \\
& \le C t(1+t)^{-\frac{n}{2 \sigma}(\frac{1}{r}-\frac{1}{2})-\frac{k-\tilde{k}}{2 \sigma}-2+\frac{1}{2 \sigma} }  
\| \nabla^{\tilde{k}}_{x} g \|_{r}, 
\end{split}
\end{equation*}
which is the desired estimate \eqref{eq:6.7}.
Next, we show \eqref{eq:6.8}.
Again we use \eqref{eq:5.1}, \eqref{eq:4.14} and \eqref{eq:5.2} 
with $C_{0}=c$, $\alpha= 2 \sigma$, $\beta= 4 \sigma-2+k-\tilde{k}$ for the first factor, and with $C_{0}=c$, $\alpha= 2 \sigma$, $\beta=4 \sigma-3 + k- \tilde{k} $ for the second factor
to obtain 
\begin{equation*} 
\begin{split}
& \left\| 
\nabla^{k}_{x} \left(
K_{3L}(t,\xi)
- \mathcal{F}^{-1} \left[e^{-\frac{\nu t |\xi|^{2 \sigma}}{2}}
\frac{\sin (t |\xi|)}{|\xi|} \chi_{L} \right] \ast g
\right)
\right\|_{2} \\
& \le C t \| e^{-c (1+t) |\xi|^{2 \sigma}} |\xi|^{k+4 \sigma -2} \chi_{L}\hat{g} \|_{2}
+ C \| e^{-c (1+t) |\xi|^{2 \sigma}} |\xi|^{4 \sigma-3} \chi_{L}\hat{g} \|_{2} \\
& \le C (t \| e^{-c (1+t) |\xi|^{2 \sigma}} |\xi|^{k-\tilde{k}+4 \sigma -2} \chi_{L} \|_{\frac{2r}{2-r}}
+ \| e^{-c (1+t) |\xi|^{2 \sigma}} |\xi|^{k-\tilde{k}+4 \sigma -3} \chi_{L} \|_{\frac{2r}{2-r}})
\| |\xi|^{\tilde{k}} \hat{g} \|_{r'} \\
& \le C  ((1+t)^{-\frac{n}{2 \sigma}(\frac{1}{r}-\frac{1}{2}) -1-\frac{k-\tilde{k}}{2 \sigma} +\frac{1}{\sigma}}
+ (1+t)^{-\frac{n}{2 \sigma}(\frac{1}{r}-\frac{1}{2}) -2-\frac{k-\tilde{k}}{2 \sigma} +\frac{3}{2\sigma}})
\| \nabla_{x}^{\tilde{k}} g \|_{r} \\
& \le C (1+t)^{-\frac{n}{2 \sigma}(\frac{1}{r}-\frac{1}{2}) -1-\frac{k-\tilde{k}}{2 \sigma} +\frac{1}{\sigma}}
\| \nabla^{\tilde{k}}_{x} g \|_{r}, 
\end{split}
\end{equation*}
where we have just used the fact that $-1+\frac{1}{2\sigma}<0$ in the last inequality.
This establishes \eqref{eq:6.8}, and the proof is complete. 
\end{proof}
%%%%%%
%%%%%%%%%%%%%
\begin{lem} \label{lem:6.4}
Let $n \ge 1$, $k \ge \tilde{k} \ge 0$, $1 \le r \le 2$, 
$\nu>0$ and $\sigma \in (\displaystyle{\frac{1}{2}},1]$. 
Then, there exist $C>0$ and $c>0$ such that 
\begin{equation} 
\begin{split}
\label{eq:6.9}
& \left\| 
\nabla^{k}_{x} \left(
 \partial_{t} K_{1L}(t)g
+
\nabla_{x}
\mathcal{F}^{-1} \left[
e^{-\frac{\nu t |\xi|^{2 \sigma}}{2}} \sin (t |\xi|) \chi_{L} 
\right] \ast g
\right)
\right\|_{2} \\
& \le C(1+t)^{-\frac{n}{2 \sigma}(\frac{1}{r}-\frac{1}{2})-1-\frac{k-\tilde{k}}{2 \sigma}}
\| \nabla_{x}^{\tilde{k}}g \|_{r}, 
\end{split}
\end{equation}
\begin{equation} 
\begin{split}
& \left\| 
\nabla^{k}_{x} \left(
\partial_{t}
K_{3L}(t)g
- \mathcal{F}^{-1} \left[
e^{-\frac{\nu t |\xi|^{2 \sigma}}{2}} \cos (t |\xi|)\chi_{L} 
\right] \ast g
\right)
\right\|_{2} \\
& \le C (1+t)^{-\frac{n}{2 \sigma}(\frac{1}{r}-\frac{1}{2})-1+\frac{1}{2 \sigma}-\frac{k-\tilde{k}}{2 \sigma}}
\| \nabla_{x}^{\tilde{k}}g \|_{r},
 \label{eq:6.10}
\end{split}
\end{equation}
where $K_{1L}(t)g$ and $K_{3L}(t)g$ are defined by \eqref{eq:2.10}.
\end{lem}
%
%%%%%
\begin{proof}
In order to show \eqref{eq:6.9},
we simply apply \eqref{eq:4.19}, \eqref{eq:5.1} and \eqref{eq:5.2} with 
$C_{0}=c$, $\alpha= 2 \sigma$, $\beta=4 \sigma+ k- \tilde{k}$ for the first factor, 
and with 
$C_{0}=c$, $\alpha= 2 \sigma$, $\beta= 2 \sigma+k- \tilde{k}$ for the second factor.
Then we see 
\begin{equation*} 
\begin{split}
& \left\| 
\nabla^{k}_{x} \left(
 \partial_{t} K_{1L}(t)g
+
\nabla_{x}
\mathcal{F}^{-1} \left[
e^{-\frac{\nu t |\xi|^{2 \sigma}}{2}} \sin (t |\xi|) \chi_{L} 
\right] \ast g
\right)
\right\|_{2} \\
& \le C t \|e^{-c (1+t) |\xi|^{2 \sigma}} |\xi|^{4 \sigma+k}  \chi_{L}\hat{g} \|_{2}
+ C \|e^{-c (1+t) |\xi|^{2 \sigma}} |\xi|^{2 \sigma+k}\chi_{L}\hat{g} \|_{2}\\ 
& \le C (t \|e^{-c (1+t) |\xi|^{2 \sigma}} |\xi|^{4 \sigma+k-\tilde{k} } \|_{\frac{2r}{2-r}}
+ C \|e^{-c (1+t) |\xi|^{2 \sigma}} |\xi|^{2 \sigma+k-\tilde{k}}\chi_{L}\|_{\frac{2r}{2-r}} )
\| |\xi|^{\tilde{k}} \hat{g} \|_{r'} \\ 
& \le C(1+t)^{-\frac{n}{2 \sigma}(\frac{1}{r}-\frac{1}{2})-1-\frac{k-\tilde{k}}{2 \sigma}}
\| \nabla_{x}^{\tilde{k}}g \|_{r}, 
\end{split}
\end{equation*}
which is the desired estimate \eqref{eq:6.9}.
The same proof is also valid for \eqref{eq:6.10}.
Indeed, 
we again apply \eqref{eq:4.20}, \eqref{eq:5.1} and \eqref{eq:5.2} with 
$C_{0}=c$, $\alpha= 2 \sigma$, $\beta= 4 \sigma-1+k- \tilde{k}$ for the first factor, 
and with 
$C_{0}=c$, $\alpha= 2 \sigma$, $\beta= 2 \sigma-1+k- \tilde{k}$ for the second factor.
Then we see 
\begin{equation*} 
\begin{split}
&
\left\| 
\nabla^{k}_{x} \left(
\partial_{t}
K_{3L}(t)g
- \mathcal{F}^{-1} \left[
e^{-\frac{\nu t |\xi|^{2 \sigma}}{2}} \cos (t |\xi|)\chi_{L} 
\right] \ast g
\right)
\right\|_{2} 
 \\
& \le C t \|e^{-c (1+t) |\xi|^{2 \sigma}} |\xi|^{4 \sigma-1+k}  \chi_{L}\hat{g} \|_{2}
+ C \|e^{-c (1+t) |\xi|^{2 \sigma}} |\xi|^{2+k}\chi_{L}\hat{g} \|_{2}\\ 
& \le C (t \|e^{-c (1+t) |\xi|^{2 \sigma}} |\xi|^{4 \sigma-1+k-\tilde{k} } \|_{\frac{2r}{2-r}}
+ C \|e^{-c (1+t) |\xi|^{2 \sigma}} |\xi|^{2 \sigma-1+k-\tilde{k}}\chi_{L}\|_{\frac{2r}{2-r}} )
\| |\xi|^{\tilde{k}} \hat{g} \|_{r'} \\ 
& \le C( t(1+t)^{-\frac{n}{2 \sigma}(\frac{1}{r}-\frac{1}{2})-2+\frac{1}{2 \sigma}-\frac{k-\tilde{k}}{2 \sigma}}
+ 
(1+t)^{-\frac{n}{2 \sigma}(\frac{1}{r}-\frac{1}{2})-1+\frac{1}{2\sigma}-\frac{k-\tilde{k}}{2 \sigma}})
\| \nabla_{x}^{\tilde{k}}g \|_{r} \\
& \le C (1+t)^{-\frac{n}{2 \sigma}(\frac{1}{r}-\frac{1}{2})-1+\frac{1}{2 \sigma}-\frac{k-\tilde{k}}{2 \sigma}}
\| \nabla_{x}^{\tilde{k}}g \|_{r}, 
\end{split}
\end{equation*}
which is the desired estimate \eqref{eq:6.10}, and the proof is now complete.
\end{proof}
%%%%%%%%%%%%
%%%%%%%%%%%%
%%%%
The error factor for $\sigma \in (\displaystyle{\frac{1}{2}}, 1]$ is estimated as follow.  
%%%%%%%
\begin{lem} \label{lem:6.5}
Let $n \ge 3$, $k \ge \tilde{k} \ge 0$, $1 \le r \le 2$, 
$\nu>0$ and $\sigma \in (\displaystyle{\frac{1}{2}},1]$. 
Then, there exist $C>0$ and $c>0$ such that 
\begin{align} 
\label{eq:6.11}
& \left\| 
\nabla^{k}_{x} \left(
\mathcal{F}^{-1} \left[
e^{-\frac{\nu t |\xi|^{2 \sigma}}{2}} \cos (t |\xi|) (\chi_{M}+\chi_{H}) \right] \ast g
\right)
\right\|_{2} \le C e^{-ct} t^{-\frac{n}{2 \sigma}(\frac{1}{r}-\frac{1}{2})-\frac{k-\tilde{k}}{2 \sigma} }
\| \nabla_{x}^{\tilde{k}} g \|_{r}, \\
& \left\| 
\nabla^{k}_{x} \left(
\mathcal{F}^{-1} \left[
e^{-\frac{\nu t |\xi|^{2 \sigma}}{2}} \sin (t |\xi|) (\chi_{M}+\chi_{H}) \right] \ast g
\right)
\right\|_{2} \le C e^{-ct} t^{-\frac{n}{2 \sigma}(\frac{1}{r}-\frac{1}{2})-\frac{k-\tilde{k}}{2 \sigma} }
\| \nabla_{x}^{\tilde{k}} g \|_{r}, \label{eq:6.12} \\
& \left\| 
\nabla^{k}_{x} \left(
\mathcal{F}^{-1} \left[e^{-\frac{\nu t |\xi|^{2 \sigma}}{2}}
\frac{\sin (t |\xi|)}{|\xi|} (\chi_{M}+\chi_{H}) \right] \ast g
\right)
\right\|_{2} \le C e^{-ct} t^{-\frac{n}{2 \sigma}(\frac{1}{r}-\frac{1}{2})-\frac{k-\tilde{k}}{2 \sigma} }
\| \nabla_{x}^{(\tilde{k}-1)_{+}} g \|_{r}, \label{eq:6.13} 
\end{align}
where $(k-1)_{+}=\max\{ k-1,0\}$.
\end{lem}
\begin{proof}
Since the same proof works for the estimates \eqref{eq:6.11} - \eqref{eq:6.13}, we only show \eqref{eq:6.11}.
\begin{align*} 
& \left\| 
\nabla^{k}_{x} \left(
\mathcal{F}^{-1} \left[
e^{-\frac{\nu t |\xi|^{2 \sigma}}{2}} \cos (t |\xi|) (\chi_{M}+\chi_{H}) \right] \ast g
\right)
\right\|_{2} \\
& \le C \| e^{-ct |\xi|^{2 \sigma}} |\xi|^{k} (\chi_{M}+\chi_{H}) \hat{g} \|_{2} \\
& \le C \| e^{-ct |\xi|^{2 \sigma}} |\xi|^{k-\tilde{k}} (\chi_{M}+\chi_{H}) \|_{\frac{2r}{2-r}}
\| |\xi|^{\tilde{k}} \hat{g} \|_{r'} \\
& \le C e^{-ct} t^{-\frac{n}{2 \sigma}(\frac{1}{r}-\frac{1}{2})-\frac{k-\tilde{k}}{2 \sigma} }
\| \nabla_{x}^{\tilde{k}} g \|_{r}, 
\end{align*}
which implies \eqref{eq:6.11}.
\end{proof}
%

%%%%%%%%%%%%%%%%%%%%%%%%%%%%%%%%%%%%%%%%%%%%%%%%%%%%%%%%%%%%%%%%%%%%%%%%%%%%%%%%%%%%%%%%%%%%%%%%%%%%%%%%%%%%%%%%%%%%%%%%%%%%%%
\subsection{Decay properties and approximation formulas for the evolution operators}
By summarizing estimates obtained in subsection 6.2 we arrive at the estimates for the evolution operators.
First, we mention a series of estimates for $J_{j}(t)g$ for $j=1,2,3,4$.
\begin{prop} \label{prop:6.6}
Let $n \ge 1$,  $\ell=0,1$, $k \ge \tilde{k}_{1} \ge 0$,
$k + \ell \ge \tilde{k}_{2} \ge 0$, $1 \le r_{1}, r_{2} \le 2$, 
$\nu>0$ and $\sigma \in (0, \displaystyle{\frac{1}{2}})$. 
Then it holds that
\begin{equation*} 
\begin{split}
%\label{eq:6.14}
%
\left\| 
\partial_{t}^{\ell} 
\nabla_{x}^{k}
J_{1}(t)g
\right\|_{2}
& \le C(1+t)^{-\frac{n}{2(1-\sigma)}(\frac{1}{r_{1}}-\frac{1}{2})- \ell- \frac{k-\tilde{k}_{1}}{2(1-\sigma)}}
\| \nabla^{\tilde{k}_{1}}_{x} g \|_{r_{1}} \\
& \ 
+ C e^{-ct} t^{-\frac{n}{2 \sigma}(\frac{1}{r_{2}}-\frac{1}{2})-\frac{k+\ell-\tilde{k}_{2}}{2 \sigma} }
\| \nabla^{\tilde{k}_{2}}_{x} g\|_{r_{2}},
\end{split}
\end{equation*}
\begin{equation*} 
\begin{split}
\left\| 
\partial_{t}^{\ell} 
\nabla_{x}^{k}
J_{3}(t)g
\right\|_{2}
& \le C(1+t)^{-\frac{n}{2\sigma}(\frac{1}{r_{1}}-\frac{1}{2})
-\frac{1-2\sigma}{\sigma} - \ell- \frac{k-\tilde{k}_{1}}{2\sigma}}
\| \nabla^{\tilde{k}_{1}}_{x} g \|_{r_{1}} \\
& + C e^{-ct} t^{-\frac{n}{2 \sigma}(\frac{1}{r_{2}}-\frac{1}{2})-\frac{k+\ell-\tilde{k}_{2}}{2 \sigma} }
\| \nabla^{\tilde{k}_{2}}_{x} g\|_{r_{2}}, 
%\label{eq:6.15}
\end{split}
\end{equation*}
where $J_{j}(t)g$ for $j=1,3$ are defined by \eqref{eq:2.4}.
\end{prop}
\begin{prop} \label{prop:6.7}
Let $n \ge 2$,  $\ell=0,1$, $k \ge \tilde{k}_{1} \ge 0$,
$k + \ell \ge \tilde{k}_{2} \ge 0$, $1 \le r_{1}, r_{2} \le 2$, 
$\nu>0$ and $\sigma \in (0, \displaystyle{\frac{1}{2}})$. 
Then it holds that
\begin{equation*} 
\begin{split}
\left\|
\partial_{t}^{\ell} 
\nabla_{x}^{k}
J_{2}(t)g 
\right\|_{2} 
& \le C(1+t)^{-\frac{n}{2(1-\sigma)}(\frac{1}{r_{1}}-\frac{1}{2})
+\frac{\sigma}{1-\sigma} - \ell- \frac{k-\tilde{k}_{1}}{2(1-\sigma)}}
\| \nabla^{\tilde{k}_{1}}_{x} g \|_{r_{1}} \\
& +
C e^{-ct} t^{-\frac{n}{2 \sigma}(\frac{1}{r_{2}}-\frac{1}{2})-\frac{k+\ell-\tilde{k}_{2}}{2 \sigma} }
\| \nabla^{(\tilde{k_{2}}-1)_{+}}_{x} g\|_{r_{2}}, 
%\label{eq:6.16} 
\end{split}
\end{equation*}
\begin{equation*} 
\begin{split}
\left\|
\partial_{t}^{\ell} 
\nabla_{x}^{k}
J_{4}(t)g 
\right\|_{2} 
& \le C(1+t)^{-\frac{n}{2\sigma}(\frac{1}{r_{1}}-\frac{1}{2})+1- \ell- \frac{k-\tilde{k}_{1}}{2\sigma}}
\| \nabla^{\tilde{k}_{1}}_{x} g \|_{r_{1}} \\
& +
C e^{-ct} t^{-\frac{n}{2 \sigma}(\frac{1}{r_{2}}-\frac{1}{2})-\frac{k+\ell-\tilde{k}_{2}}{2 \sigma} }
\| \nabla^{(\tilde{k_{2}}-1)_{+}}_{x} g\|_{r_{2}}, 
%\label{eq:6.17}
\end{split}
\end{equation*}
where $J_{j}(t)g$ for $j=2,4$ are defined by \eqref{eq:2.4}.
\end{prop}
\begin{proof}[Proof of Propositions \ref{prop:6.6} and \ref{prop:6.7}]
It is obvious form the combinations of \eqref{eq:5.4} - \eqref{eq:5.9}.
\end{proof}
The following corollaries are easy consequences of Propositions \ref{prop:6.6} and \ref{prop:6.7}. 
\begin{cor} \label{cor:6.8}
Let $n \ge 1$, $k \ge 0$, $\ell=0,1$, 
$\nu>0$ and $\sigma \in (0, \displaystyle{\frac{1}{2}})$. 
Then it holds that
\begin{align}
\label{eq:6.18}
& \left\| 
\partial_{t}^{\ell} 
\nabla_{x}^{k}
J_{1}(t)g
\right\|_{2}
\le C(1+t)^{-\frac{n}{4(1-\sigma)}- \ell- \frac{k}{2(1-\sigma)}} \| g \|_{1}  
+ C e^{-ct} 
\| \nabla^{k+ \ell}_{x} g\|_{2},\\
& \left\| 
\partial_{t}^{\ell} 
\nabla_{x}^{k}
J_{3}(t)g
\right\|_{2}
 \le C(1+t)^{-\frac{n}{4\sigma}
-\frac{1-2\sigma}{\sigma} - \ell- \frac{k}{2\sigma}}
\| g \|_{1} + C e^{-ct} \| \nabla_{x}^{k+ \ell} g\|_{2}, \label{eq:6.19} 
\end{align}
where $J_{j}(t)g$ for $j=1,3$ are defined by \eqref{eq:2.4}.
\end{cor}
\begin{cor} \label{cor:6.9}
Let $n \ge 2$, $k \ge 0$, $\ell=0,1$, 
$\nu>0$ and $\sigma \in (0, \displaystyle{\frac{1}{2}})$. 
Then it holds that
\begin{align}
& \left\|
\partial_{t}^{\ell} 
\nabla_{x}^{k}
J_{2}(t)g 
\right\|_{2} 
\le C(1+t)^{-\frac{n}{4(1-\sigma)}
+\frac{\sigma}{1-\sigma} - \ell- \frac{k}{2(1-\sigma)}}
\| g \|_{1} 
+
C e^{-ct} 
\| \nabla^{(k+\ell-1)_{+}}_{x} g\|_{2}, \label{eq:6.20} \\
& \left\|
\partial_{t}^{\ell} 
\nabla_{x}^{k}
J_{4}(t)g 
\right\|_{2} 
\le C(1+t)^{-\frac{n}{4\sigma}+1- \ell- \frac{k}{2\sigma}}
\| g \|_{1} 
+
C e^{-ct} 
\| \nabla^{(k+\ell-1)_{+}}_{x} g\|_{2}, \label{eq:6.21}
\end{align}
where $J_{j}(t)g$ for $j=2,4$ are defined by \eqref{eq:2.4}.
\end{cor}
\begin{proof}
In Propositions \ref{prop:6.6} and \ref{prop:6.7}, it suffices to choose $\tilde{k}_{1}=0$, $\tilde{k}_{2}=k +\ell$,
$r_{1}=1$ and $r_{2}=2$ to obtain \eqref{eq:6.18} - \eqref{eq:6.21}.
\end{proof}
\begin{rem} \label{rem:6.10}
It is worth pointing out that because of Corollaries \ref{cor:6.8} and \ref{cor:6.9}, the leading factor of the asymptotic behavior 
as $t \to \infty$ is given by $\| \partial_{t}^{\ell} \nabla_{x}^{k}J_{2}(t)g\|_{2}$.
Indeed, 
roughly speaking,
Corollaries \ref{cor:6.8} and \ref{cor:6.9} suggest that we can regard $J_{j}(t)g$ as
\begin{equation} \label{eq:6.22}
\begin{split}
& \left\| 
\partial_{t}^{\ell} 
\nabla_{x}^{k}
J_{1}(t)g
\right\|_{2}
\sim (1+t)^{-\frac{n}{4(1-\sigma)}- \ell- \frac{k}{2(1-\sigma)}}  ,\\
& \left\|
\partial_{t}^{\ell} 
\nabla_{x}^{k}
J_{2}(t)g 
\right\|_{2} 
\sim (1+t)^{-\frac{n}{4(1-\sigma)}
+\frac{\sigma}{1-\sigma} - \ell- \frac{k}{2(1-\sigma)}}, \\
& \left\| 
\partial_{t}^{\ell} 
\nabla_{x}^{k}
J_{3}(t)g
\right\|_{2}
\sim (1+t)^{-\frac{n}{4\sigma}
-\frac{1-2\sigma}{\sigma} - \ell- \frac{k}{2\sigma}},\\
& \left\|
\partial_{t}^{\ell} 
\nabla_{x}^{k}
J_{4}(t)g 
\right\|_{2} 
\sim (1+t)^{-\frac{n}{4\sigma}+1- \ell- \frac{k}{2\sigma}}
\end{split}
\end{equation}
for large $t$, and we have  
$\left\|
\partial_{t}^{\ell} 
\nabla_{x}^{k}
J_{j}(t)g 
\right\|_{2} \le \left\|
\partial_{t}^{\ell} 
\nabla_{x}^{k}
J_{2}(t)g 
\right\|_{2}$ for $j=1,3,4$, formally under the assumption $\sigma \in (0, \displaystyle{\frac{1}{2}})$. 

We also recall that the solution $u(t)$ to \eqref{eq:1.1} with $u_{1} \equiv 0$ is given by 
$$
u(t) = J_{1}(t) u_{0} +J_{3}(t) u_{0}.
$$
Similar arguments can be applied to this case, and we see  
$\left\|
\partial_{t}^{\ell} 
\nabla_{x}^{k}
J_{3}(t)g 
\right\|_{2} \le \left\|
\partial_{t}^{\ell} 
\nabla_{x}^{k}
J_{1}(t)g 
\right\|_{2}$ formally under the assumption $\sigma \in (0, \displaystyle{\frac{1}{2}})$ again.
Thus we also find that the leading factor of the case $u_{1} \equiv 0$ is given by $J_{1}(t)g$. 
\end{rem}
By observations in Remark \ref{rem:6.10}, 
we need to construct the approximation of the evolution operators $J_{j}(t)g$ for $j=1,2$.
\begin{prop} \label{prop:6.11}

Let $n \ge 1$,  $\ell=0,1$, $k \ge \tilde{k}_{1} \ge 0$,
$k + \ell \ge \tilde{k}_{2} \ge 0$, $1 \le r_{1}, r_{2} \le 2$, 
$\nu>0$ and $\sigma \in (0, \displaystyle{\frac{1}{2}})$. 
Then it holds that
\begin{equation} 
\begin{split}
\label{eq:6.23}
& \left\| 
\partial_{t}^{\ell} 
\nabla_{x}^{k}
\left(
J_{1}(t)g
-
\mathcal{F}^{-1}
\left[
e^{
-\frac{t}{\nu}
|\xi|^{2(1-\sigma)}
}
\right] \ast g
\right)
\right\|_{2} \\
&\le C(1+t)^{-\frac{n}{2(1-\sigma)}(\frac{1}{r_{1}}-\frac{1}{2})
-\frac{1-2\sigma}{1-\sigma} - \ell- \frac{k-\tilde{k}_{1}}{2(1-\sigma)}}
\| \nabla^{\tilde{k_{1}}}_{x} g \|_{r_{1}} \\
& \ 
+ C e^{-ct} t^{-\frac{n}{2 \sigma}(\frac{1}{r_{2}}-\frac{1}{2})-\frac{k+\ell-\tilde{k}_{2}}{2 \sigma} }
\| \nabla^{\tilde{k}_{2}}_{x} g\|_{r_{2}},
\end{split}
\end{equation}
\begin{equation} 
\begin{split}
&\left\|
\partial_{t}^{\ell} 
\nabla_{x}^{k}
\left(
J_{2}(t)g 
-
\mathcal{F}^{-1} 
\left[
\frac{
e^{
-\frac{t}{\nu}
|\xi|^{2(1-\sigma)}
}
}{ \nu |\xi|^{2 \sigma}}
\right] \ast g 
\right) 
\right\|_{2} \\
& \le C(1+t)^{-\frac{n}{2(1-\sigma)}(\frac{1}{r_{1}}-\frac{1}{2})
-\frac{1-3\sigma}{1-\sigma} - \ell- \frac{k-\tilde{k_{1}}}{2(1-\sigma)}}
\| \nabla^{\tilde{k}_{1}}_{x} g \|_{r_{1}} \\
& +
C e^{-ct} t^{-\frac{n}{2 \sigma}(\frac{1}{r_{2}}-\frac{1}{2})-\frac{k+\ell-\tilde{k}_{2}}{2 \sigma} }
\| \nabla^{(\tilde{k_{2}}-1)_{+}}_{x} g\|_{r_{2}}, \label{eq:6.24} 
\end{split}
\end{equation}
where $J_{j}(t)g$ for $j=1,2$ are defined by \eqref{eq:2.4}.
\end{prop}
\begin{proof}
Combining \eqref{eq:6.1}, \eqref{eq:6.5} and \eqref{eq:5.8} gives \eqref{eq:6.23}.
We apply this argument again, with \eqref{eq:6.1}, \eqref{eq:6.5} and \eqref{eq:5.8} replaced by
\eqref{eq:6.2}, \eqref{eq:6.6} and \eqref{eq:5.9}, to obtain \eqref{eq:6.24},
which completes the proof of the statement of Proposition \ref{prop:6.11}.
\end{proof}
The following corollary is an easy consequence of Propositions \ref{prop:6.6} and \ref{prop:6.7}.
On the other hand, it is important to determine the leading factor of the large time behavior. 
\begin{cor} \label{cor:6.12}
Let $n \ge 1$, $k \ge 0$, $\ell=0,1$, $\nu>0$ and $\sigma \in (0, \displaystyle{\frac{1}{2}})$. 
Then it holds that
\begin{equation*}
\begin{split}
%\label{eq:6.25}
%
& \left\| 
\partial_{t}^{\ell} 
\nabla_{x}^{k}
\left(
J_{1}(t)g
-
\mathcal{F}^{-1}
\left[
e^{
-\frac{t}{\nu}
|\xi|^{2(1-\sigma)}
}
\right] \ast g
\right)
\right\|_{2} \\
& \le C(1+t)^{-\frac{n}{4(1-\sigma)}
-\frac{1-2\sigma}{1-\sigma} - \ell- \frac{k}{2(1-\sigma)}}
\| g \|_{1} + C e^{-ct} \| \nabla^{k + \ell}_{x} g\|_{2}, 
\end{split}
\end{equation*}
\begin{equation}
\begin{split}
&\left\|
\partial_{t}^{\ell} 
\nabla_{x}^{k}
\left(
J_{2}(t)g 
-
\mathcal{F}^{-1} 
\left[
\frac{
e^{
-\frac{t}{\nu}
|\xi|^{2(1-\sigma)}
}
}{ \nu |\xi|^{2 \sigma}}
\right] \ast g 
\right) 
\right\|_{2} \\
& \le C(1+t)^{-\frac{n}{4(1- \sigma)}
-\frac{1-3\sigma}{1-\sigma} - \ell- \frac{k}{2(1-\sigma)}}
\| g \|_{1} +C e^{-ct} \| \nabla_{x}^{(k+ \ell-1)_{+}} g\|_{2}, 
\label{eq:6.26} 
\end{split}
\end{equation}
where $J_{j}(t)g$ for $j=1,2$ is defined by \eqref{eq:2.4}.
\end{cor}
\begin{proof}
Proposition \ref{prop:6.11} with $k_{1}=0$, $k_{2}= k+ \ell$, $r_{1}=1$ and $r_{2}=2$ 
gives Corollary \ref{cor:6.12}.
\end{proof}
Secondly, we summarize the estimates for $K_{j}(t)g$ for $j=1,2,3$.
\begin{prop} \label{prop:6.13}
Let $n \ge 1$, $\ell=0,1$, $k+ \ell \ge \tilde{k}_{1} \ge 0$, 
$k \ge \tilde{k}_{2} \ge 0$, $1 \le r_{1}, r_{2} \le 2$, 
$\nu>0$ and $\sigma \in (\displaystyle{\frac{1}{2}},1]$. 
Then it holds that
\begin{equation*} 
\begin{split}
%\label{eq:6.27}
%
\left\| 
\partial_{t}^{\ell} 
\nabla_{x}^{k}
K_{1}(t)g
\right\|_{2}
& \le C(1+t)^{-\frac{n}{2 \sigma }(\frac{1}{r_{1}}-\frac{1}{2})- \frac{\ell+ k-\tilde{k}_{1}}{2\sigma}}
\| \nabla^{\tilde{k}_{1}}_{x} g \|_{r_{1}} \\
&  + C e^{-ct} 
t^{-\frac{n}{2(1-\sigma)}(\frac{1}{r_{2}}-\frac{1}{2})-\frac{k-\tilde{k}_{2}}{2(1- \sigma)} }
\| \nabla^{(\tilde{k}_{2}+2 \ell(1 -\sigma))}_{x} g\|_{r_{2}}, 
\end{split}
\end{equation*}
\begin{equation*} 
%\label{eq:6.28}
\begin{split}
\left\|
\partial_{t}^{\ell} 
\nabla_{x}^{k}
K_{2}(t)g 
\right\|_{2} 
& \le C(1+t)^{-\frac{n}{2 \sigma }(\frac{1}{r_{1}}-\frac{1}{2})- \frac{\ell+ k-\tilde{k}_{1}+1}{2\sigma}}
\| \nabla^{\tilde{k}_{1}}_{x} g \|_{r_{1}} , \\
&  + C e^{-ct} 
t^{-\frac{n}{2(1-\sigma)}(\frac{1}{r_{2}}-\frac{1}{2})-\frac{k-\tilde{k}_{2}}{2 (1-\sigma)} }
\| \nabla^{(\tilde{k}_{2}+2\ell(1-\sigma))}_{x} g\|_{r_{2}}, 
\end{split}
\end{equation*}
where $K_{j}(t)g$ for $j=1,2$ are defined by \eqref{eq:2.9}. 
\end{prop}
As is mentioned in Remark \ref{rem:6.10}, 
the Fourier multiplier of $\K_{3L}(t,\xi)$ has a singularity in the sense of the $L^{2}$ integrability in the
low dimensional case.
Thus we have to state the result for $\K_{3L}(t,\xi)$ separately from $K_{1}(t)g$ and $K_{2}(t)g$. 
\begin{prop} \label{prop:6.14}
Let $n \ge 3$, $\ell=0,1$, $\max\{ \ell+ k -1, 0 \}\ge \tilde{k}_{1} \ge 0$, 
$k \ge \tilde{k}_{2} \ge 0$,
$1 \le r_{1}, r_{2} \le 2$, 
$\nu>0$ and $\sigma \in (\displaystyle{\frac{1}{2}},1]$. 
Then it holds that
\begin{equation*}
\begin{split} 
%
%\label{eq:6.29}
%
\left\| 
\partial_{t}^{\ell} 
\nabla_{x}^{k}
K_{3}(t)g
\right\|_{2}
& \le C(1+t)^{-\frac{n}{2 \sigma }(\frac{1}{r_{1}}-\frac{1}{2})- \frac{\ell+ k-\tilde{k}_{1}-1}{2\sigma}}
\| \nabla^{\tilde{k}_{1}}_{x} g \|_{r_{1}} \\
& +C  e^{-ct} 
t^{-\frac{n}{2(1-\sigma)}(\frac{1}{r_{2}}-\frac{1}{2})-\frac{k-\tilde{k}_{2}}{2 (1-\sigma)} }
\| \nabla^{(\tilde{k}_{2}-2 \sigma(1-\ell))_{+}}_{x} g\|_{r_{2}}, 
\end{split}
\end{equation*}
where $K_{3}(t)g$ is defined by \eqref{eq:2.9}, and 
$(\tilde{k}_{2}-2\sigma(1-\ell))_{+} = \max\{\tilde{k}_{2}-2 \sigma (1- \ell), 0\}$. 
\end{prop}
%%%%%
%%%%%
\begin{proof}[Proof of Propositions \ref{prop:6.13} and \ref{prop:6.14}]
The proof is the direct consequence of \eqref{eq:5.10} - \eqref{eq:5.14}.
\end{proof}
As easy consequences of Propositions \ref{prop:6.13} and \ref{prop:6.14}, 
we obtain the following estimates, which suggest the leading factor of the asymptotic behavior 
of $K_{j}(t)g$, $j=1,2,3$, as $t \to \infty$.   
%%%%%%%%%
\begin{cor} \label{cor:6.15}
Let $n \ge 1$, $k \ge 0$, $\ell=0,1$, $1 \le r \le 2$, 
$\nu>0$ and $\sigma \in (\displaystyle{\frac{1}{2}},1]$. 
Then it holds that
\begin{align} 
\label{eq:6.30}
& \left\| 
\partial_{t}^{\ell} 
\nabla_{x}^{k}
K_{1}(t)g
\right\|_{2} \le C(1+t)^{-\frac{n}{4 \sigma}- \frac{\ell+ k}{2\sigma}}
\| g \|_{1} + C e^{-ct} 
\| \nabla^{(k+2\ell(1-\sigma))}_{x} g\|_{2}, \\
& \left\|
\partial_{t}^{\ell} 
\nabla_{x}^{k}
K_{2}(t)g 
\right\|_{2} 
\le C(1+t)^{-\frac{n}{4 \sigma}- \frac{\ell+ k+1}{2\sigma}}
\| g \|_{1}
+ C e^{-ct} 
\| \nabla^{(k+2\ell(1-\sigma))}_{x} g\|_{2}, \label{eq:6.31}
\end{align}
where $K_{j}(t)g$ for $j=1,2$ are defined by \eqref{eq:2.9}. 
\end{cor}
\begin{cor} \label{cor:6.16}
Let $n \ge 3$, $\ell=0,1$, 
$\nu>0$ and $\sigma \in (\displaystyle{\frac{1}{2}},1]$. 
Then it holds that
%%%%%%
%
\begin{equation}
\begin{split} 
\label{eq:6.32}
\left\| 
\partial_{t}^{\ell} 
\nabla_{x}^{k}
K_{3}(t)g
\right\|_{2} \le C(1+t)^{-\frac{n}{4 \sigma }- \frac{\ell+ k-1}{2\sigma}}
\| g \|_{1} +C  e^{-ct} 
\| \nabla^{(k-2 \sigma(1-\ell))_{+}}_{x} g\|_{2}, 
\end{split}
\end{equation}
where $K_{3}(t)g$ is defined by \eqref{eq:2.9}, and 
$(\tilde{k}_{2}-2\sigma(1-\ell))_{+} = \max\{\tilde{k}_{2}-2 \sigma (1- \ell), 0\}$. 
\end{cor}
%%%%%
\begin{proof}[Proof of Corollaries \ref{cor:6.15} and \ref{cor:6.16}]
To obtain \eqref{eq:6.30} and \eqref{eq:6.31}, 
we apply Proposition \ref{prop:6.13} with $r_{1}=1$, $r_{2}=2$, $k_{1}=0$ and $k_{2}=k$. 
We apply this argument again with Proposition \ref{prop:6.13} replaced by Proposition \ref{prop:6.14},
to get \eqref{eq:6.32}. 
\end{proof}
%%%%%%%%%%%%%%%%%%%%
\begin{rem}
We remark that the same reasoning as \eqref{eq:6.22} can be applied to the case for Corollaries 
\ref{cor:6.15} and \ref{cor:6.16}.
Namely, roughly speaking, the estimates \eqref{eq:6.30} and \eqref{eq:6.32} tell us that 
\begin{equation*} 
%\label{eq:6.33}
\begin{split}
& \left\| 
\partial_{t}^{\ell} 
\nabla_{x}^{k}
K_{1}(t)g
\right\|_{2}
\sim (1+t)^{-\frac{n}{4 \sigma}- \frac{\ell+ k}{2\sigma}} ,\\
& \left\|
\partial_{t}^{\ell} 
\nabla_{x}^{k}
K_{2}(t)g 
\right\|_{2} 
\sim
(1+t)^{-\frac{n}{4 \sigma}- \frac{\ell+ k+1}{2\sigma}}, \\
& \left\| 
\partial_{t}^{\ell} 
\nabla_{x}^{k}
K_{3}(t)g
\right\|_{2}
\sim (1+t)^{-\frac{n}{4 \sigma }- \frac{\ell+ k-1}{2\sigma}},
\end{split}
\end{equation*}
and we see 
$\left\| 
\partial_{t}^{\ell} 
\nabla_{x}^{k}
K_{j}(t)g
\right\|_{2} \le \left\| 
\partial_{t}^{\ell} 
\nabla_{x}^{k}
K_{3}(t)g
\right\|_{2}$ for $j=1,2$.
On the other hand, 
if $u_{1} \equiv 0$, we have $u(t) = K_{1}(t) u_{1}$, and so we need to obtain the approximation formulas of $K_{1}(t)g$ and $K_{3}(t)g$.

\end{rem}
%%%%%%%%%%%%%%%
It is easy to see from Lemmas \ref{lem:4.3} and \ref{lem:4.4} that 
$K_{1}(t)g$ and $K_{3}(t)g$ are approximated by  
$\mathcal{F}^{-1} \left[
e^{-\frac{\nu t |\xi|^{2 \sigma}}{2}} \cos (t |\xi|)  \right] \ast g$
and 
$\mathcal{F}^{-1} \left[e^{-\frac{\nu t |\xi|^{2 \sigma}}{2}}
\frac{\sin (t |\xi|)}{|\xi|} \right] \ast g$
respectively.
On the other hand, the approximation of $\partial_{t} K_{1}(t)g$ and $\partial_{t} K_{3}(t)g$ are not given by 
$\partial_{t} \mathcal{F}^{-1} \left[
e^{-\frac{\nu t |\xi|^{2 \sigma}}{2}} \cos (t |\xi|)  \right] \ast g$
and 
$\partial_{t} \mathcal{F}^{-1} \left[e^{-\frac{\nu t |\xi|^{2 \sigma}}{2}}
\frac{\sin (t |\xi|)}{|\xi|} \right] \ast g$.
%%%%%%%
\begin{prop} \label{prop:6.18}
Let $n \ge 1$, $k \ge \tilde{k}_{1}, \tilde{k}_{2} \ge 0$, $1 \le r_{1}, r_{2} \le 2$, 
$\nu>0$ and $\sigma \in (\displaystyle{\frac{1}{2}},1]$. 
Then it holds that
\begin{equation} 
\begin{split}
\label{eq:6.34}
& \left\| 
\nabla^{k}_{x} \left(
 K_{1}(t)g
-
\mathcal{F}^{-1} \left[
e^{-\frac{\nu t |\xi|^{2 \sigma}}{2}} \cos (t |\xi|)  \right] \ast g
\right)
\right\|_{2} \\
& \le C 
(1+t)^{-\frac{n}{2 \sigma}(\frac{1}{r_{1}}-\frac{1}{2})-\frac{k-\tilde{k}_{1}}{2 \sigma}-1+\frac{1}{2 \sigma} }  
\| \nabla^{\tilde{k}_{1}}_{x} g \|_{r_{1}} \\ 
&  + C e^{-ct} 
t^{-\frac{n}{2(1-\sigma)}(\frac{1}{r_{2}}-\frac{1}{2})-\frac{k-\tilde{k}_{2}}{2 (1-\sigma)} }
\| \nabla^{\tilde{k}_{2}}_{x} g\|_{r_{2}}, 
\end{split}
\end{equation}
\begin{equation} 
\begin{split}
& \left\| 
\nabla^{k}_{x} \left(
K_{3}(t)g
- \mathcal{F}^{-1} \left[e^{-\frac{\nu t |\xi|^{2 \sigma}}{2}}
\frac{\sin (t |\xi|)}{|\xi|} \right] \ast g
\right)
\right\|_{2} \\
& \le C (1+t)^{-\frac{n}{2 \sigma}(\frac{1}{r}-\frac{1}{2}) -1-\frac{k-\tilde{k}}{2 \sigma} +\frac{1}{\sigma}}
\| \nabla_{x}^{\tilde{k}} g \|_{r} \\
& +C  e^{-ct} 
t^{-\frac{n}{2(1-\sigma)}(\frac{1}{r_{2}}-\frac{1}{2})-\frac{k-\tilde{k}_{2}}{2 (1-\sigma)} }
\| \nabla^{(\tilde{k}_{2}-2 \sigma)_{+}}_{x} g\|_{r_{2}}, 
 \label{eq:6.35}
\end{split}
\end{equation}
\begin{equation} 
\begin{split}
\label{eq:6.36}
& \left\| 
\nabla^{k}_{x} \left(
 \partial_{t} K_{1}(t)g
+
\nabla_{x}
\mathcal{F}^{-1} \left[
e^{-\frac{\nu t |\xi|^{2 \sigma}}{2}} \sin (t |\xi|)
\right] \ast g
\right)
\right\|_{2} \\
& \le C(1+t)^{-\frac{n}{2 \sigma}(\frac{1}{r_{1}}-\frac{1}{2})-1-\frac{k-\tilde{k}_{1}}{2 \sigma}}
\| \nabla_{x}^{\tilde{k}_{1}}g \|_{r_{1}}
+ C e^{-ct} 
t^{-\frac{n}{2(1-\sigma)}(\frac{1}{r_{2}}-\frac{1}{2})-\frac{k-\tilde{k}_{2}}{2 (1-\sigma)} }
\| \nabla^{(\tilde{k}_{2}+2(1-\sigma))}_{x} g\|_{r_{2}}, 
\end{split}
\end{equation}
\begin{equation} 
\begin{split}
& \left\| 
\nabla^{k}_{x} \left(
\partial_{t}
K_{3}(t)g
- \mathcal{F}^{-1} \left[
e^{-\frac{\nu t |\xi|^{2 \sigma}}{2}} \cos (t |\xi|)
\right] \ast g
\right)
\right\|_{2} \\
& \le C (1+t)^{-\frac{n}{2 \sigma}(\frac{1}{r_{1}}-\frac{1}{2})-1+\frac{1}{2 \sigma}-\frac{k-\tilde{k}_{1}}{2 \sigma}}
\| \nabla_{x}^{\tilde{k}_{1}}g \|_{r_{1}} +C  e^{-ct} 
t^{-\frac{n}{2(1-\sigma)}(\frac{1}{r_{2}}-\frac{1}{2})-\frac{k-\tilde{k}_{2}}{2 (1-\sigma)} }
\| \nabla^{\tilde{k}_{2}}_{x} g\|_{r_{2}}, 
 \label{eq:6.37}
\end{split}
\end{equation}
where $K_{j}(t)g$, $j=1,3$ are defined by \eqref{eq:2.9}.
\end{prop}
%%%%%%%
\begin{proof}
The estimates \eqref{eq:5.13}, \eqref{eq:6.7} and \eqref{eq:6.11} mean \eqref{eq:6.34}.
Similarly, by \eqref{eq:5.14}, \eqref{eq:6.8} and \eqref{eq:6.13}
we get \eqref{eq:6.35}.
\eqref{eq:6.36} and \eqref{eq:6.37} are shown by the same manner.%
\end{proof}
%%%%%%%%
We can now rephrase Proposition \ref{prop:6.18} as follows.
%%%%%%%
\begin{cor} \label{cor:6.19}
Let $n \ge 1$, $k \ge 0$, $\ell=0,1$, $1 \le r \le 2$, 
$\nu>0$ and $\sigma \in (\displaystyle{\frac{1}{2}},1]$. 
Then it holds that
\begin{equation} 
\begin{split}
\label{eq:6.38}
& \left\| 
\nabla^{k}_{x} \left(
 K_{1}(t)g
-
\mathcal{F}^{-1} \left[
e^{-\frac{\nu t |\xi|^{2 \sigma}}{2}} \cos (t |\xi|)  \right] \ast g
\right)
\right\|_{2} \\
& \le C 
(1+t)^{-\frac{n}{4 \sigma}-\frac{k}{2 \sigma}-2+\frac{1}{2 \sigma} }  
\| g \|_{1}+ C e^{-ct} 
\| \nabla^{k}_{x} g\|_{2}, 
\end{split}
\end{equation}
\begin{equation} 
\begin{split}
& \left\| 
\nabla^{k}_{x} \left(
K_{3}(t)g - \mathcal{F}^{-1} \left[e^{-\frac{\nu t |\xi|^{2 \sigma}}{2}}
\frac{\sin (t |\xi|)}{|\xi|} \right] \ast g
\right)
\right\|_{2} \\
&  \le C (1+t)^{-\frac{n}{4 \sigma}-1-\frac{k}{2 \sigma} +\frac{1}{\sigma}}
\|  g \|_{1}  +C  e^{-ct} \| \nabla^{(k-2 \sigma)_{+}}_{x} g\|_{2}, 
 \label{eq:6.39}
\end{split}
\end{equation}
\begin{equation} 
\begin{split}
\label{eq:6.40}
& \left\| 
\nabla^{k}_{x} \left(
 \partial_{t} K_{1}(t)g
+
\nabla_{x}
\mathcal{F}^{-1} \left[
e^{-\frac{\nu t |\xi|^{2 \sigma}}{2}} \sin (t |\xi|)
\right] \ast g
\right)
\right\|_{2} \\
& \le C(1+t)^{-\frac{n}{4 \sigma}-1-\frac{k}{2 \sigma}}
\|g \|_{1}
+ C e^{-ct} 
\| \nabla^{(k+2(1-\sigma))}_{x} g\|_{2}, 
\end{split}
\end{equation}
\begin{equation} 
\begin{split}
& \left\| 
\nabla^{k}_{x} \left(
\partial_{t}
K_{3}(t)g
- \mathcal{F}^{-1} \left[
e^{-\frac{\nu t |\xi|^{2 \sigma}}{2}} \cos (t |\xi|)
\right] \ast g
\right)
\right\|_{2} \\
& \le C (1+t)^{-\frac{n}{4 \sigma}-1+\frac{1}{2 \sigma}-\frac{k}{2 \sigma}}
\| g \|_{1} +C  e^{-ct} 
\| \nabla^{k}_{x} g\|_{2}, 
 \label{eq:6.41}
\end{split}
\end{equation}
where $K_{j}(t)g$, $j=1,3$ are defined by \eqref{eq:2.9}.
\end{cor}
\begin{proof}
In Proposition \ref{prop:6.18}, we choose $r_{1}=1$, $r_{2}=2$, $k_{1}=0$ and $k_{2} =2$, so that we have 
\eqref{eq:6.38} - \eqref{eq:6.41}.
\end{proof}
%%%%%%
Finally, we deal with the case $\sigma = \displaystyle{\frac{1}{2}}$.
%%%%
\begin{prop} \label{prop:6.20}
Let $n \ge 1$, $\ell=0,1$, $k+ \ell \ge \tilde{k}_{1},\tilde{k}_{2} \ge 0$, 
$k \ge \tilde{k}_{3},\tilde{k}_{4} \ge 0$,$1 \le r_{1}, r_{2} \le 2$, 
$\nu>0$ and $\sigma=\displaystyle{\frac{1}{2}}$. 
Then it holds that
\begin{equation} 
\begin{split}
\label{eq:6.42}
\left\| 
\partial_{t}^{\ell} 
\nabla_{x}^{k}
\tilde{J}_{1}(t)g
\right\|_{2}
+
\left\| 
\partial_{t}^{\ell} 
\nabla_{x}^{k}
\tilde{J}_{2}(t)g
\right\|_{2}
& \le C(1+t)^{-n(\frac{1}{r_{1}}-\frac{1}{2})-(\ell+ k-\tilde{k}_{1})}
\| \nabla^{\tilde{k}_{1}}_{x} g \|_{r_{1}} \\
&  + C e^{-ct} 
t^{-n(\frac{1}{r_{2}}-\frac{1}{2})-(\ell+k-\tilde{k}_{2}) }
\| \nabla^{\tilde{k}_{2}}_{x} g\|_{r_{2}}, 
\end{split}
\end{equation}
\begin{equation} 
\begin{split}
\label{eq:6.43}
\left\| 
\partial_{t}^{\ell} 
\nabla_{x}^{k}
\tilde{K}_{1}(t)g
\right\|_{2}
+
\left\| 
\partial_{t}^{\ell} 
\nabla_{x}^{k}
\tilde{K}_{2}(t)g
\right\|_{2}
& \le C(1+t)^{-n(\frac{1}{r_{1}}-\frac{1}{2})-(\ell+ k-\tilde{k}_{1})}
\| \nabla^{\tilde{k}_{1}}_{x} g \|_{r_{1}} \\
&  + C e^{-ct} 
t^{-n(\frac{1}{r_{2}}-\frac{1}{2})-(\ell+k-\tilde{k}_{2}) }
\| \nabla^{\tilde{k}_{2}}_{x} g\|_{r_{2}}, 
\end{split}
\end{equation}
\begin{equation} 
\begin{split}
\label{eq:6.44}
\left\| 
\partial_{t}^{\ell} 
\nabla_{x}^{k}
E_{1}(t)g
\right\|_{2}
+
\left\| 
\partial_{t}^{\ell} 
\nabla_{x}^{k}
E_{2}(t)g
\right\|_{2}
& \le C(1+t)^{-n(\frac{1}{r_{1}}-\frac{1}{2})-(\ell+ k-\tilde{k}_{1})}
\| \nabla^{\tilde{k}_{1}}_{x} g \|_{r_{1}} \\
&  + C e^{-ct} 
t^{-n(\frac{1}{r_{2}}-\frac{1}{2})-(\ell+k-\tilde{k}_{2}) }
\| \nabla^{\tilde{k}_{2}}_{x} g\|_{r_{2}}, 
\end{split}
\end{equation}
\begin{equation} 
\begin{split}
\label{eq:6.45}
\left\| 
\partial_{t}^{\ell} 
\nabla_{x}^{k}
\tilde{J}_{3}(t)g
\right\|_{2}
& \le Ct^{1- \ell} (1+t)^{-n(\frac{1}{r_{1}}-\frac{1}{2})-(k-\tilde{k}_{3})}
\| \nabla^{\tilde{k}_{3}}_{x} g \|_{r_{1}} \\
&  + C e^{-ct} 
t^{(1-\ell)-n(\frac{1}{r_{2}}-\frac{1}{2})-(k-\tilde{k}_{4}) }
\| \nabla^{\tilde{k}_{4}}_{x} g\|_{r_{2}}, 
\end{split}
\end{equation}
\begin{equation} 
\begin{split}
\label{eq:6.46}
\left\| 
\partial_{t}^{\ell} 
\nabla_{x}^{k}
\tilde{K}_{3}(t)g
\right\|_{2}
& \le Ct^{1- \ell} (1+t)^{-n(\frac{1}{r_{1}}-\frac{1}{2})-(k-\tilde{k}_{3})}
\| \nabla^{\tilde{k}_{3}}_{x} g \|_{r_{1}} \\
&  + C e^{-ct} 
t^{(1-\ell)-n(\frac{1}{r_{2}}-\frac{1}{2})-(k-\tilde{k}_{4}) }
\| \nabla^{\tilde{k}_{4}}_{x} g\|_{r_{2}}, 
\end{split}
\end{equation}
\begin{equation} 
\begin{split}
\label{eq:6.47}
\left\| 
\partial_{t}^{\ell} 
\nabla_{x}^{k}
E_{3}(t)g
\right\|_{2}
& \le Ct^{1- \ell} (1+t)^{-n(\frac{1}{r_{1}}-\frac{1}{2})-(k-\tilde{k}_{3})}
\| \nabla^{\tilde{k}_{3}}_{x} g \|_{r_{1}} \\
&  + C e^{-ct} 
t^{(1-\ell)-n(\frac{1}{r_{2}}-\frac{1}{2})-(k-\tilde{k}_{4}) }
\| \nabla^{\tilde{k}_{4}}_{x} g\|_{r_{2}}, 
\end{split}
\end{equation}
where $\tilde{J}_{j}(t)g$ and $\tilde{K}_{j}(t)g$ for $j=1,2,3$ are defined by \eqref{eq:2.16}
and
$E_{j}(t)g$ for $j=1,2,3$ are defined by \eqref{eq:2.19}. 
\end{prop}
\begin{proof}
It is easy to see that \eqref{eq:6.42} - \eqref{eq:6.47} are obtained by similar way, so we only show the proof \eqref{eq:6.42}.
Now we use \eqref{eq:4.25}, \eqref{eq:5.1}, \eqref{eq:5.2} 
with $C_{0}=c$, $s=1+t$, $\alpha=1$ and $\beta=\ell+k-\tilde{k}_{1}$, and
\eqref{eq:5.3} with $C_{0}=c$, $s=t$, $\alpha=1$ and $\beta=\ell+k-\tilde{k}_{2}$ to observe
\begin{align*}
& \| \partial_{t}^{\ell} \nabla_{x}^{k} \tilde{J}_{1}(t)  g \|_{2} 
+ \| \partial_{t}^{\ell} \nabla_{x}^{k} \tilde{J}_{2}(t) g \|_{2} 
\\
& \le C \| e^{-c(1+t)|\xi|} |\xi|^{\ell+k-\tilde{k}_{1}} \chi_{L}  |\xi|^{\tilde{k}_{1}}  \hat{g} \|_{2} 
+ C\| e^{-ct|\xi|} |\xi|^{\ell+k-\tilde{k}_{2}} (\chi_{M}+\chi_{H})  |\xi|^{\tilde{k}_{2}}  \hat{g} \|_{2} \\
& \le C \| e^{-c(1+t)|\xi|} |\xi|^{\ell+k-\tilde{k}_{1}} \chi_{L}  \|_{\frac{2r_{1}}{2-r_{1}}} 
\| |\xi|^{\tilde{k}_{1}}  \hat{g} \|_{r_{1}'} 
+ C
\| e^{-ct|\xi|} |\xi|^{\ell+k-\tilde{k}_{2}} (\chi_{M}+\chi_{H})  \|_{\frac{2r_{2}}{2-r_{2}}}  
\| |\xi|^{\tilde{k}_{2}}  \hat{g} \|_{2} \\
& \le C(1+t)^{-n(\frac{1}{r_{1}}-\frac{1}{2})-(\ell+ k-\tilde{k}_{1})}
\| \nabla^{\tilde{k}_{1}}_{x} g \|_{r_{1}} 
+ C e^{-ct} 
t^{-n(\frac{1}{r_{2}}-\frac{1}{2})-(\ell+k-\tilde{k}_{2}) }
\| \nabla^{\tilde{k}_{2}}_{x} g\|_{r_{2}}, 
\end{align*}
which implies \eqref{eq:6.42}.  
\end{proof}
To determine the asymptotic behavior of the solution of \eqref{eq:1.1} with $\sigma=\displaystyle{\frac{1}{2}}$, 
the following estimates are useful. 
\begin{cor} \label{cor:6.21}
Let $n \ge 1$, $\ell=0,1$, $k \ge 0$, 
$\nu>0$ and $\sigma=\displaystyle{\frac{1}{2}}$. 
Then it holds that
\begin{equation} 
\begin{split}
\label{eq:6.48}
\left\| 
\partial_{t}^{\ell} 
\nabla_{x}^{k}
\tilde{J}_{1}(t)g
\right\|_{2}
+
\left\| 
\partial_{t}^{\ell} 
\nabla_{x}^{k}
\tilde{J}_{2}(t)g
\right\|_{2}
& \le C(1+t)^{-\frac{n}{2}-(\ell+ k)}
\| g \|_{1}  + C e^{-ct} 
\| \nabla_{x}^{\ell+k} g\|_{2}, 
\end{split}
\end{equation}
\begin{equation} 
\begin{split}
\label{eq:6.49}
\left\| 
\partial_{t}^{\ell} 
\nabla_{x}^{k}
\tilde{K}_{1}(t)g
\right\|_{2}
+
\left\| 
\partial_{t}^{\ell} 
\nabla_{x}^{k}
\tilde{K}_{2}(t)g
\right\|_{2}
& \le C(1+t)^{-\frac{n}{2}-(\ell+ k)}
\| g \|_{1} + C e^{-ct} 
\| \nabla_{x}^{\ell+k} g\|_{2}, 
\end{split}
\end{equation}
\begin{equation} 
\begin{split}
\label{eq:6.50}
\left\| 
\partial_{t}^{\ell} 
\nabla_{x}^{k}
E_{1}(t)g
\right\|_{2}
+
\left\| 
\partial_{t}^{\ell} 
\nabla_{x}^{k}
E_{2}(t)g
\right\|_{2}
& \le C(1+t)^{-\frac{n}{2}-(\ell+ k)}
\| g \|_{1} + C e^{-ct} 
\| \nabla_{x}^{\ell+k} g\|_{2}, 
\end{split}
\end{equation}
\begin{equation} 
\begin{split}
\label{eq:6.51}
\left\| 
\partial_{t}^{\ell} 
\nabla_{x}^{k}
\tilde{J}_{3}(t)g
\right\|_{2}
& \le Ct^{1- \ell} (1+t)^{-\frac{n}{2}-k}
\| g \|_{1} 
+ C e^{-ct} 
\| \nabla^{(k+\ell-1)_{+}}_{x} g\|_{2}, 
\end{split}
\end{equation}
\begin{equation} 
\begin{split}
\label{eq:6.52}
\left\| 
\partial_{t}^{\ell} 
\nabla_{x}^{k}
\tilde{K}_{3}(t)g
\right\|_{2}
& \le Ct^{1- \ell} (1+t)^{-\frac{n}{2}-k}
\| g \|_{1} 
+ C e^{-ct} 
\| \nabla^{(k+\ell-1)_{+}}_{x} g\|_{2}, 
\end{split}
\end{equation}
\begin{equation} 
\begin{split}
\label{eq:6.53}
\left\| 
\partial_{t}^{\ell} 
\nabla_{x}^{k}
E_{3}(t)g
\right\|_{2}
& \le Ct^{1- \ell} (1+t)^{-\frac{n}{2}-k}
\| g \|_{1} 
+ C e^{-ct} 
\| \nabla^{(k+\ell-1)_{+}}_{x} g\|_{2}, 
\end{split}
\end{equation}
where 
$(k+\ell-1)_{+}:=\max \{k+\ell-1, 0 \}$, 
$\tilde{J}_{j}(t)g$ and $\tilde{K}_{j}(t)g$ for $j=1,2,3$ are defined by \eqref{eq:2.16}
and
$E_{j}(t)g$ for $j=1,2,3$ are defined by \eqref{eq:2.19}. 
\end{cor}
\begin{proof}
The proof is a direct consequence of Proposition \ref{prop:6.20} with $r_{1}=1$, 
$r_{2}=2$, $\tilde{k}_{1}=0$, $\tilde{k}_{2} = \ell+k$, $\tilde{k}_{3} = 0$ and $\tilde{k}_{4} = (k+l-1)_{+}$.
\end{proof}
%%%%%%%%
%%%%%%%%%%%%%%
%%%%%%%%%%%%%%%%%%%%%%%%%%%%%%%%%%%%%%%%%Section7%%%%%%%%%%%%%%%%%%%%%%%%%%%%%%%%%%%%%%%%%%%%%%%%%%%%%%%%%%%%%%%%%%%%%%%%%%%%%%%%%%%%%%%%%

\section{Asymptotic profiles of solutions}
In this section, we first rephrase the results in section \ref{sec:6} as the solution to the Cauchy problem \eqref{eq:1.1}.
We also observe the upper bound of the decay order of the solution to the Cauchy problem \eqref{eq:1.1}.
Secondly we state the asymptotic expansion formula for the convolution type function in a general setting.
Finally, we complete the proof of main results by a combination of the results obtained in this section.

\subsection{Solution of the Cauchy problem \eqref{eq:1.1}}
We can now reformulate the estimates stated in Proposition \ref{prop:6.6} - Corollary \ref{cor:6.19}
 as a property of the solution to Cauchy problem \eqref{eq:1.1}.
For $\sigma \in (0, 1/2)$, 
one has the following.
\begin{prop} \label{prop:7.1}
Let $n \ge 2$, $\sigma \in (0, \displaystyle{\frac{1}{2}})$, 
$\nu >0$ and $k_{0} \ge 0$. 
If $(u_{0}, u_{1}) \in (H^{k_{0}+1} \cap L^{1}) \times (H^{k_{0}} \cap L^{1})$,
then there exists a unique solution $u(t) \in C([0,\infty);H^{k_{0}+1} ) \cap C^{1}([0, \infty);H^{k_{0}})$ satisfying 
\begin{align*} 
%\label{eq:7.1}
%
\| \partial_{t}^{\ell}  \nabla^{k}_{x} u(t) \|_{2} 
\le C (1+t)^{-\frac{n}{4 (1-\sigma)}+\frac{\sigma}{1-\sigma}-\ell-\frac{k}{2(1-\sigma)}}
\end{align*}
for $\ell=0,1$, $k \in [0,k_{0}+1]$ and $k+ \ell \le k_{0}+1$.
\end{prop}
When $\sigma = \displaystyle{\frac{1}{2}}$, we can deal with all dimension $n \ge 1$ (cf. \cite{D}, \cite{DR}, \cite{NR}). 
\begin{prop} \label{prop:7.2}
Let $n \ge 1$, $\sigma=\displaystyle{\frac{1}{2}}$, 
$\nu >0$ and $k_{0} \ge 0$. 
If $(u_{0}, u_{1}) \in (H^{k_{0}+1} \cap L^{1}) \times (H^{k_{0}} \cap L^{1})$,
then there exists a unique solution $u(t) \in C([0,\infty);H^{k_{0}+1} ) \cap C^{1}([0, \infty);H^{k_{0}})$ satisfying 
\begin{equation*} 
%\label{eq:7.2}
\| \partial_{t}^{\ell}  \nabla^{k}_{x} u(t) \|_{2} 
\le C t^{1-\ell} (1+t)^{-\frac{n}{2}-k}
\end{equation*}
for $\ell=0,1$, $k \in [0,k_{0}+1]$ and $k+ \ell \le k_{0}+1$.
\end{prop}
The following result implies that the solution to \eqref{eq:1.1} with $\sigma \in (\frac{1}{2}, 1]$ has
a different dissipative structure from the one with $\sigma \in (0, \frac{1}{2}]$, as was pointed out in references 
(see \cite{DR}, \cite{I}, \cite{ITY},
\cite{NR}, \cite{P} and \cite{S}).
\begin{prop} \label{prop:7.3}
Let $n \ge 3$, $\sigma \in (\displaystyle{\frac{1}{2}},1]$, 
$\nu >0$ and $k_{0} \ge 0$. 
If $(u_{0}, u_{1}) \in (H^{k_{0}+2 \sigma} \cap L^{1}) \times (H^{k_{0}} \cap L^{1})$,
then there exists a unique solution $u(t) \in C([0,\infty);H^{k_{0}+2 \sigma} ) \cap C^{1}([0, \infty);H^{k_{0}})$ satisfying 
\begin{align*} 
%\label{eq:7.3}
%
\| \partial_{t}^{\ell} \nabla^{k}_{x} u(t) \|_{2} \le C (1+t)^{-\frac{n}{4\sigma}- \frac{\ell+k-1}{2 \sigma}}
\end{align*}
for $\ell=0,1$ and $k \in [0, k_{0}+ 2 \sigma]$ and $k+ \ell \le k_{0}+2 \sigma$.
\end{prop}
As an easy consequence of Propositions \ref{prop:7.1} - \ref{prop:7.3}, 
one has a decay property of the solution to problem \eqref{eq:1.1} with a special initial data $u_{1} \equiv 0$.
\begin{prop} \label{prop:7.4}
Let $n \ge 1$, $\sigma \in (0,1]$, 
$\nu >0$ and $k_{0} \ge 0$. 
If $(u_{0}, u_{1}) \in (H^{k_{0}+1} \cap L^{1}) \times (H^{k_{0}} \cap L^{1})$,
then there exists a unique solution $u(t) \in C([0,\infty);H^{k_{0}+1} ) \cap C^{1}([0, \infty);H^{k_{0}})$ satisfying 
\begin{align*} 
%\label{eq:7.4}
%
\| \partial_{t}^{\ell}  \nabla^{k}_{x} u(t) \|_{2}  \le
\begin{cases}
& C (1+t)^{-\frac{n}{4 (1-\sigma)}-\ell-\frac{k}{2(1-\sigma)}}  \quad (\sigma \in (0, \frac{1}{2}) ),\\
& C (1+t)^{-\frac{n}{2}-(\ell+ k)} \quad (\sigma =\frac{1}{2}), \\
&  C (1+t)^{-\frac{n}{4\sigma}- \frac{\ell+k}{2 \sigma}} \quad (\sigma \in (\frac{1}{2}, 1] )
\end{cases}
\end{align*}
for $\ell=0,1$ and $k \in [0, k_{0}]$.
\end{prop}
\begin{proof}[Proof of Propositions \ref{prop:7.1} - \ref{prop:7.4}]
Proposition \ref{prop:7.1} is shown by \eqref{eq:6.18} - \eqref{eq:6.21} 
together with the solution formula \eqref{eq:2.3}.
By the similar way,  
Propositions \ref{prop:7.2} and \ref{prop:7.3} are proved by \eqref{eq:6.48} - \eqref{eq:6.53}
with \eqref{eq:2.20}, and 
\eqref{eq:6.30} - \eqref{eq:6.32} with 
\eqref{eq:2.8}, respectively.
Propositions \ref{prop:7.1} - \ref{prop:7.3} with $u_{1} \equiv 0$ directly yield Proposition \ref{prop:7.4}.
This completes the proof of Propositions \ref{prop:7.1} - \ref{prop:7.4}.
\end{proof}

\subsection{General approximation formula}
In this subsection, we show an approximation formula for a convolution type of function in terms of the integral kernel with a suitable constant.
The following proposition plays an important role to prove our main results.
\begin{prop} \label{prop:7.5}
Let $n \ge 1$ and $k \ge 0$.
Suppose that $g \in L^{1}(\R^{n})$ and the smooth function $\K(t,x)$ satisfies  
\begin{align} \label{eq:7.5}
\| \nabla_{x}^{k} \K(t) \|_{2} \le Ct^{-\gamma_{1}}, \quad
\| \nabla_{x}^{k+1}\K(t) \|_{2} \le C t^{-\gamma_{2}} 
\end{align}
with some $C>0$, and $0<\gamma_{1}<\gamma_{2}$.
Then,
it holds that
\begin{equation} \label{eq:7.6}
\begin{split}
& \left\| 
\nabla_{x}^{k} \left(
\K(t) \ast g
- m
\K(t, \cdot) 
\right)
\right\|_{2} =o(t^{-\gamma_{1}})
\end{split}
\end{equation}
as $t \to \infty$, 
where $m:= \displaystyle{\int_{\R^{n}}}g(y)\,dy$.
\end{prop}
\begin{proof}[Proof of Proposition \ref{prop:7.5}]
First, we observe that
the mean value theorem yields 
\begin{align} \label{eq:7.7}
\K(t,x-y) 
- \K(t,x) = (-y) \cdot \nabla_{x} 
\K(t,x-\theta y)
\end{align}
for some $\theta \in (0,1)$.
Now, we decompose the integrand in the left hand side of \eqref{eq:7.6} by using \eqref{eq:7.7} 
such as
\begin{equation*}
\begin{split}
(\K(t) \ast g)(x)
- 
m \K(t, x) 
& =\K(t) \ast g
- 
(\int_{\R^{n}} g(y)\,dy)
\K(t, x) \\
& = \int_{|y| \le t^{
\frac{\gamma_{2}-\gamma_{1}
}
{2
}
} }
(-y) \cdot \nabla_{x} 
\K(t,x-\theta y) 
g(y)
dy\,  \\
& \ \ +
\int_{|y| \ge t^{
\frac{\gamma_{2}-\gamma_{1}
}
{2
}
} }
(
\K(t,x-y) 
- \K(t,x)
)
g(y)
dy.
\end{split}
\end{equation*}
Thus, applying $\nabla_{x}^{k}$ and taking $L^{2}$ norm in both sides, 
we easily see that 
\begin{equation} \label{eq:7.8}
\begin{split}
& \left\| \nabla_{x}^{k}
\left(
\K(t) \ast g
- 
m
\K(t, x)
\right)
 \right\|_{2} \\
& \le C 
t^{
\frac{\gamma_{2}-\gamma_{1}
}
{2
}
}
\| \nabla^{k+1}_{x} 
\K(t,x-\theta y) 
\|_{L^{2}_{x}} 
\int_{|y| \le t^{
\frac{\gamma_{2}-\gamma_{1}
}
{2
}
} }
|
g(y)
|
dy\ \\
& \ +(\|
\nabla^{k+1}_{x} \K(t,x-y) \|_{L^{2}_{x}} 
+\|
\nabla^{k+1}_{x} \K(t) \|_{L^{2}_{x}} 
)
\int_{|y| \ge t^{
\frac{\gamma_{2}-\gamma_{1}
}
{2
}
} }
|g(y)| 
dy \\
& \le C t^{-\frac{\gamma_{1} + \gamma_{2}}{2}}
\|g \|_{1}
+
C t^{-\gamma_{1}}
\int_{|y| \ge t^{
\frac{\gamma_{2}-\gamma_{1}
}
{2
}
} }
|g(y)| 
dy\, ,
\end{split}
\end{equation}
where we have just used \eqref{eq:7.5}.
Here, 
we note that $\gamma_{2}-\gamma_{1}>0$ and $g \in L^{1}$, so that we see 
\begin{align} \label{eq:7.9}
\int_{|y| \ge t^{
\frac{\gamma_{2}-\gamma_{1}
}
{2
}
} }
|g(y)| 
dy
\to 0
\end{align}
as $t \to \infty$.
Combining \eqref{eq:7.8} and \eqref{eq:7.9},
we have arrived at the desired estimate \eqref{eq:7.6}.
\end{proof}

\subsection{Kernel estimates}
In this subsection,
we show some estimates for the Fourier multipliers in $L^{2}$-based Sobolev spaces to apply Proposition 
\ref{prop:7.5}, which will be used in the proof of Theorems \ref{thm:1.2} and \ref{thm:1.3}. 

\begin{lem} \label{lem:7.6}
Assume \eqref{eq:1.6}.
Let $\ell=0,1$ and $k \ge 0$.
Then, 
$\mathcal{G}_{\sigma,\nu}(t,\xi)$ defined by \eqref{eq:1.4} satisfies 
$\partial_{t}^{\ell} \mathcal{G}_{\sigma,\nu}(t,\xi) \in L^{1} \cap L^{2}(\R^{n})$, and further $\partial_{t}^{\ell} \mathcal{F}^{-1}[\mathcal{G}_{\sigma,\nu}(t,\xi)]$ are
 well-defined for $\ell=0,1$.
Moreover, the following estimates hold:
\begin{align} \label{eq:7.10}
& \|\partial_{t}^{\ell} \nabla_{x}^{k} \mathcal{F}^{-1}[\mathcal{G}_{\sigma,\nu}(t,\xi)] \| _{2}
=
Ct^{-\gamma_{\sigma,k}-\ell}
\end{align}
for $\sigma \in (0,\displaystyle{\frac{1}{2}}]$ and
\begin{align} \label{eq:7.11}
& Ct^{-\gamma_{\sigma,k}}
\le \| \nabla^{k}_{x} \mathcal{F}^{-1}[\mathcal{G}_{\sigma,\nu}(t,\xi)] \| _{2}
\le Ct^{-\gamma_{\sigma,k}}, \\
& 
Ct^{-\gamma_{\sigma,k}-\frac{1}{2 \sigma}} \le \|\partial_{t}  \nabla^{k}_{x} 
\mathcal{F}^{-1}
[
e^{-\frac{\nu t |\xi|^{2 \sigma}}{2}
} \cos(t |\xi|)
] 
\| _{2}
\le 
Ct^{-\gamma_{\sigma,k}-\frac{1}{2 \sigma}}, 
\label{eq:7.12}
\end{align}
for $\sigma \in (\displaystyle{\frac{1}{2}},1]$.
\end{lem}
\begin{proof}[Proof of Lemma \ref{lem:7.6}]
The assumption \eqref{eq:1.16} ensures that 
$|\xi|^{-2\sigma} \in L^{1} \cap L^{2}(|\xi| \le 1)$ for $n \ge2$ with $\sigma \in (0, \frac{1}{2})$ and
$|\xi|^{-1} \in L^{1} \cap L^{2}(|\xi| \le 1)$ for $n \ge 3$ with $\sigma \in (\frac{1}{2},1]$.
Then we easily see that $\mathcal{F}^{-1}[\mathcal{G}_{\sigma,\nu}(t,x)]$ is well-defined, since 
$\mathcal{G}_{\sigma,\nu}(t,\xi)$ decays exponentially when $|\xi|$ is large. 
On the other hand, 
the direct calculation shows that 
\begin{align*}
\partial_{t} \mathcal{G}_{\sigma,\nu}(t,\xi) =
\begin{cases}
& \dfrac{e^{-\frac{2}{\nu}t |\xi|^{2(1-\sigma)}}  |\xi|^{2(1-2 \sigma)}}{\nu^{2} }
 \quad \text{for} \ \ 0 < \sigma < \frac{1}{2}, \nu>0, \\
& e^{-\frac{\nu}{2} t |\xi|^{2 \sigma} } 
(-\nu |\xi|^{2 \sigma-1}  \sin(t |\xi|) + 2 \cos (t|\xi|) ) \quad
 \text{\rm for} \ \ \frac{1}{2} < \sigma \le 1,\ \nu>0,
\end{cases}
\end{align*}
and then we also easily see that $\partial_{t} \mathcal{F}^{-1}[\mathcal{G}_{\sigma,\nu}(t,x)]$ is well-defined.
%since $|\xi|^{2(1-2\sigma)} \le |\xi|^{-2\sigma}$ if $|\xi|$ is small 
%and $0 <\sigma < \frac{1}{2}$ 
%and $2 \sigma-1>0$ if $\frac{1}{2} <\sigma \le 1$.
For the case $\sigma=\frac{1}{2}$, the well-definedness of 
$\mathcal{F}^{-1}[\mathcal{G}_{\sigma,\nu}(t,x)]$ 
and
$\partial_{t} \mathcal{F}^{-1}[\mathcal{G}_{\sigma,\nu}(t,x)]$ 
is trivial.
We next show \eqref{eq:7.10}. Indeed, the Plancherel formula and the changing integral variable $\eta=t^{\frac{1}{2(1-\sigma)}}$ yield that
\begin{equation*}
\begin{split}
\| \partial_{t}^{\ell}\nabla_{x}^{k} \mathcal{F}^{-1}[\mathcal{G}_{\sigma,\nu}(t,\xi)] \| _{2}
& = C
\| |\xi|^{k-2 \sigma+ 2\ell (1- \sigma)} e^{-\frac{1}{\nu} t|\xi|^{2(1- \sigma)}} \| _{2} \\
& = 
Ct^{-\frac{n}{4(1-\sigma)}+\frac{\sigma}{(1-\sigma)}-\ell - \frac{k}{2(1-\sigma)}}
\| |\eta|^{k-2 \sigma+ 2\ell (1- \sigma)} e^{-\frac{1}{\nu}|\eta|^{2(1- \sigma)}} \| _{2}, 
\end{split}
\end{equation*}
which is our claim \eqref{eq:7.10}. 
Similar arguments can be applied to the case $\sigma=\frac{1}{2}$
to have  \eqref{eq:7.12}.
Finally, we deal with the case $\sigma \in (\frac{1}{2}, 1]$.
In this case, observing that 
\begin{align*}
\left| \frac{\sin(t|\xi|)}{|\xi|} \right|\le \frac{1}{|\xi|}, \quad |\cos(t|\xi|)| \le 1,
\end{align*}
we have 
\begin{equation*}
\begin{split}
 \| \nabla_{x}^{k} \mathcal{F}^{-1}[\mathcal{G}_{\sigma,\nu}(t,\xi)] \| _{2}
\le C
\| |\xi|^{k-1} e^{-c t|\xi|^{2 \sigma}} \| _{2}
\le 
Ct^{-\frac{n}{4\sigma}+\frac{1}{2\sigma} - \frac{k}{2\sigma}}, 
\end{split}
\end{equation*}
where we have used \eqref{eq:5.2} and \eqref{eq:5.3} with
$C_{0}=c$, $s=t$, $\alpha=2 \sigma$ and $\beta=k-1$.
To show the lower bound of the decay rate in \eqref{eq:7.11}, we apply the argument due to \cite{I}. By the Plancherel formula and
changing the integral variable $\eta= t^{\frac{1}{2 \sigma}} \xi$,
we see   
\begin{equation*}
\begin{split}
\| \nabla_{x}^{k} \mathcal{F}^{-1}[\mathcal{G}_{\sigma,\nu}(t,\xi)] \| _{2}^{2}
& = C \| |\xi|^{k-1}
e^{-\frac{\nu}{2} t |\xi|^{2 \sigma} } 
\sin (t |\xi|) \| _{2}^{2} \\
& = C t^{-\frac{n}{2 \sigma}-\frac{k-1}{\sigma}} 
\| |\eta|^{k-1}
e^{-\frac{\nu}{2} |\eta|^{2 \sigma} } 
\sin (t^{1-\frac{1}{2 \sigma}} |\eta|) \| _{2}^{2} \\
& = C t^{-\frac{n}{2 \sigma}-\frac{k-1}{\sigma}} 
\int_{\R^{n}}
|\eta|^{2(k-1)}
e^{-\nu |\eta|^{2 \sigma} } 
\sin^{2} (t^{1-\frac{1}{2 \sigma}} |\eta|) d \eta \\
& = \frac{1}{2} 
C t^{-\frac{n}{2 \sigma}-\frac{k-1}{\sigma}} 
\int_{\R^{n}}
|\eta|^{2(k-1)}
e^{-\nu |\eta|^{2 \sigma} } d \eta \\
& - \frac{1}{2} C t^{-\frac{n}{2 \sigma}-\frac{k-1}{\sigma}} 
\int_{\R^{n}}
|\eta|^{2(k-1)}
e^{-\nu |\eta|^{2 \sigma} } 
\cos (2 t^{1-\frac{1}{2 \sigma}} |\eta|) d \eta \\
& \ge Ct^{-\frac{n}{2 \sigma}-\frac{k-1}{\sigma}}
 -o(t^{-\frac{n}{2 \sigma}-\frac{k-1}{\sigma}})
\end{split}
\end{equation*}
as $t \to \infty$.
Indeed, 
the polar coordinate transform and the Riemann-Lebesgue theorem 
give rise to
\begin{align*}
& \lim_{t \to \infty}
\int_{\R^{n}}
|\eta|^{2(k-1)}
e^{-\nu |\eta|^{2 \sigma} } 
\cos (2 t^{1-\frac{1}{2 \sigma}} |\eta|) d \eta \\
& =C
\lim_{t \to \infty}
\int_{0}^{\infty}
\tau^{2(k-1)+n-1}
e^{-\nu \tau^{2 \sigma} } 
\cos (2 t^{1-\frac{1}{2 \sigma}} \tau) d \tau=0
\end{align*}
for $n \ge 3$ and $k \ge 0$ because of $\tau^{2(k-1)+n-1}
e^{-\nu \tau^{2 \sigma} } \in L^{1}(0, \infty)$, 
which proves \eqref{eq:7.11}.
The estimate \eqref{eq:7.12} is shown by the same way.
This completes the proof of Lemma \ref{lem:7.6}.
\end{proof}
The estimates for $\mathcal{H}_{\sigma, \nu}(t, \xi)$ are obtained more easily, 
since  $\mathcal{H}_{\sigma, \nu}(t, \xi)$ does not have singularity near $\xi=0$ 
, and decays exponentially as $\vert \xi\vert \to \infty$.
\begin{lem} \label{lem:7.7}
Let $n \ge 1$, $\ell=0,1$,
$k \ge 0$ and $\sigma \in (0, 1]$.
Then,
$\mathcal{H}_{\sigma,\nu}(t,\xi)$ defined by \eqref{eq:1.5} satisfies $\partial_{t}^{\ell} \mathcal{H}_{\sigma,\nu}(t,\xi) \in L^{1} \cap L^{2}(\R^{n})$,
 and $\partial_{t}^{\ell} \mathcal{F}^{-1}[\mathcal{H}_{\sigma,\nu}(t,\xi)]$ are
 well-defined for $\ell=0,1$.
Moreover the following estimates hold good:
for $\nu>0$,
\begin{align} \label{eq:7.13}
& \|\partial_{t}^{\ell} \nabla_{x}^{k} \mathcal{F}^{-1}[\mathcal{H}_{\sigma,\nu}(t,\xi)] \| _{2}
= 
Ct^{-\tilde{\gamma}_{\sigma,k}-\ell}
\end{align}
in the case $\sigma \in (0,\displaystyle{\frac{1}{2}}]$, and
\begin{align} \label{eq:7.14}
&  
Ct^{-\tilde{\gamma}_{\sigma,k}}
\le
\| \nabla^{k}_{x} \mathcal{F}^{-1}[\mathcal{H}_{\sigma,\nu}(t,\xi)] \| _{2}
\le 
Ct^{-\tilde{\gamma}_{\sigma,k}}, \\
&
Ct^{-\tilde{\gamma}_{\sigma,k}-\frac{1}{2 \sigma}}
\le \|\partial_{t}  \nabla^{k+1}_{x} 
\mathcal{F}^{-1}
[
e^{-\frac{\mu t |\xi|^{2 \sigma}}{2}
} \sin(t |\xi|)
] 
\| _{2}
\le 
Ct^{-\tilde{\gamma}_{\sigma,k}-\frac{1}{2 \sigma}}, 
\label{eq:7.15}
\end{align}
in the case for $\sigma \in (\displaystyle{\frac{1}{2}},1]$.
\end{lem}
\begin{proof}[Proof of Lemma \ref{lem:7.7}]
It follows from the expression of $\mathcal{H}_{\sigma,\nu}(t, \xi)$ (see \eqref{eq:1.5}) that 
\begin{align*} 
%\label{eq:7.16}
|\mathcal{H}_{\sigma,\nu}(t,\xi)| \le 
\begin{cases}
& Ce^{-\frac{2}{\nu}t |\xi|^{2(1-\sigma) }}
 \ \ \text{for} \ \ 0 < \sigma < \frac{1}{2}, \nu>0, \\
& C
e^{-\frac{\nu}{2} t|\xi| }  \ \text{for} \ \ \sigma =\frac{1}{2},\ 0<\nu<2, \\
&
(1+t |\xi|) e^{-t|\xi|} 
\ \ \text{for} \ \ \sigma =\frac{1}{2},\ \nu=2, \\
& 
e^{-c t|\xi| } \ \text{for} \ \ \sigma =\frac{1}{2},\ \nu>2, \ \\
& 
e^{-\frac{\nu}{2} t |\xi|^{2 \sigma} }
 \ \ \text{for} \ \ \frac{1}{2} < \sigma \le 1,\ \nu>0. 
\end{cases}
\end{align*}
Then by the similar way to the proof of Lemma \ref{lem:7.6},
one has the desired estimates \eqref{eq:7.13} - \eqref{eq:7.15}.
\end{proof}

\subsection{Proof of main results}
\begin{proof}[Proof of Proposition \ref{prop:1.1}]
Proposition \ref{prop:1.1} is proved by Propositions \ref{prop:7.1} - \ref{prop:7.3}. 
\end{proof}

\begin{proof}[Proof of Theorem \ref{thm:1.2}]
Since \eqref{eq:1.9} - \eqref{eq:1.11} are shown by the same way, it suffices to treat \eqref{eq:1.9} for the case $\sigma \in (0, \frac{1}{2})$.
Indeed, by using \eqref{eq:6.18}, \eqref{eq:6.19}, \eqref{eq:6.21}, \eqref{eq:6.26} and \eqref{eq:7.6} 
with 
$\gamma_{1}=\frac{n}{4(1-\sigma)}- \frac{\sigma}{1-\sigma}+\frac{k}{2(1-\sigma)}$
and
$\gamma_{2}=\frac{n}{4(1-\sigma)}- \frac{\sigma}{1-\sigma}+\frac{k+1}{2(1-\sigma)}$,
we arrive at the estimate 
\begin{equation*}
\begin{split}
&  \| \nabla_{x}^{k} (u(t) - m_{1} \mathcal{F}^{-1} [\mathcal{G}_{\sigma, \nu}(t)] ) \|_{2} \\
& \le \sum_{j=1,3} \| \nabla_{x}^{k} J_{j}(t) u_{0}  \|_{2}
+ \| \nabla_{x}^{k} J_{4}(t) u_{0}  \|_{2} \\
& + \left\|
\nabla_{x}^{k}
\left(
J_{2}(t)u_{1} 
-
\mathcal{F}^{-1} 
\left[
\frac{
e^{
-\frac{t}{\nu}
|\xi|^{2(1-\sigma)}
}
}{ \nu |\xi|^{2 \sigma}}
\right] \ast u_{1} 
\right) 
\right\|_{2} \\
& + \left\|
\nabla_{x}^{k}
\left(
\mathcal{F}^{-1} 
\left[
\frac{
e^{
-\frac{t}{\nu}
|\xi|^{2(1-\sigma)}
}
}{ \nu |\xi|^{2 \sigma}}
\right] \ast u_{1} 
-m_{1}\mathcal{F}^{-1} [\mathcal{G}_{\sigma, \nu}(t)]
\right) 
\right\|_{2} =o(t^{-\gamma_{\sigma, k}})
\end{split}
\end{equation*}
as $t \to \infty$, 
which is the desired \eqref{eq:1.9} with $\sigma \in (0, \frac{1}{2})$.  
Next, we show \eqref{eq:1.13}.
To check \eqref{eq:1.13}, 
it suffices to prove the estimate from below of \eqref{eq:1.13},
since we already have the upper bound \eqref{eq:1.7} of the decay order.
By combining \eqref{eq:1.9} with $\sigma \in (0, \frac{1}{2}]$, \eqref{eq:1.10} and \eqref{eq:7.10},
we obtain  
\begin{align*}
 \| \partial_{t}^{\ell} \nabla_{x}^{k} u(t) \|_{2}
& \ge |m_{1}|
\| \partial_{t}^{\ell} \nabla_{x}^{k}  \mathcal{F}^{-1} [\mathcal{G}_{\sigma, \nu}(t)] \|_{2} 
-
\| \partial_{t}^{\ell} \nabla_{x}^{k} (u(t) - m_{1} \mathcal{F}^{-1} [\mathcal{G}_{\sigma, \nu}(t)] ) \|_{2} \\
& = C t^{-\gamma_{\sigma,k}-\ell} -o(t^{-\gamma_{\sigma,k}-\ell})
\end{align*}
as $t \to \infty$, which is the desired estimate \eqref{eq:1.13}.
Now, again we apply this argument using \eqref{eq:1.8}, \eqref{eq:1.9} with $\sigma \in (\frac{1}{2},1]$, \eqref{eq:1.11}, and \eqref{eq:7.12}
to check \eqref{eq:1.14}. This completes the proof of Theorem 1.2.
\end{proof}
%%%
%%%
\begin{proof}[Proof of Theorem \ref{thm:1.3}]
We can now proceed analogously to the proof of Theorem \ref{thm:1.2} to conclude the statement of Theorem \ref{thm:1.3}.
We shall omit the detail.
\end{proof}
%%%%%%%%%%%%%%%%
%%%%%%%%%%%%%%%%
\vspace*{5mm}
\noindent
\textbf{Acknowledgments. }
%\smallskip
The work of the first author (R. IKEHATA) was supported in part by Grant-in-Aid for Scientific Research (C)15K04958 of JSPS.
The work of the second author (H. TAKEDA) was supported in part by Grant-in-Aid for Young Scientists (B)15K17581 of JSPS.

%%%%%%
%%%%%%%%%
%%%%%%%%%%%
\end{document}